\documentclass{article}

\usepackage{microtype}
\usepackage{graphicx}
\usepackage{subfigure}
\usepackage{booktabs} 

\usepackage{hyperref}

\usepackage[preprint]{icml2026}

\usepackage{amsmath}
\usepackage{amssymb}
\usepackage{mathtools}
\usepackage{amsthm}

\usepackage[capitalize,noabbrev]{cleveref}

\theoremstyle{plain}
\newtheorem{theorem}{Theorem}[section]
\newtheorem{proposition}[theorem]{Proposition}
\newtheorem{lemma}[theorem]{Lemma}

\theoremstyle{definition}

\theoremstyle{remark}

\usepackage[textsize=tiny]{todonotes}

\icmltitlerunning{A Spectral Boundary for Independence Learnability}

\begin{document}

\twocolumn[
\icmltitle{A Spectral Boundary for Independence Learnability}



\icmlsetsymbol{equal}{*}

\begin{icmlauthorlist}
\icmlauthor{Senhan Yao}{amss,ucas}
\icmlauthor{Yuxin Zhou}{amss,ucas}
\end{icmlauthorlist}

\icmlaffiliation{amss}{State Key Laboratory of Mathematical Sciences, Academy of Mathematics and Systems Science, Chinese
Academy of Sciences, Beijing, China}
\icmlaffiliation{ucas}{School of Mathematical Sciences, University of Chinese Academy of Sciences, Beijing, China}

\icmlcorrespondingauthor{Senhan Yao}{yaosenhan@amss.ac.cn}

\icmlkeywords{Machine Learning, ICML}

\vskip 0.3in
]



\begin{NoHyper}
\printAffiliationsAndNotice{}
\end{NoHyper}

\begin{abstract}
Learning independence is fundamental to statistical inference and machine learning, yet learning independence between entire time series by a test remains challenging when only a single observation record is available. Existing guarantees often rely on structural or mixing assumptions, while stationarity and ergodicity alone do not permit consistent testing. In this work, we identify a class-level boundary governed by shared nontrivial Koopman eigenvalues, which represent common persistent oscillatory frequencies of the two processes. When no such eigenvalues are shared, we extend shift-based rank calibration with finite-window dependence witnesses and a data-adaptive threshold to construct a deterministic, weakly consistent test with finite-sample false-positive control, without quantitative mixing assumptions. Conversely, we show that no possibly randomized test can achieve vanishing false-positive probabilities under every independent jointly totally ergodic law while consistently detecting every dependent jointly totally ergodic law with a shared nontrivial eigenvalue. This impossibility already holds when the shared point spectrum is exactly the cyclic group generated by any single prescribed irrational eigenfrequency, identifying a spectral obstruction to universal learning of full-process independence.
\end{abstract}

\section{Introduction}
\label{sec:introduction}

Independence is a central concept in probability, statistics, and machine learning, underpinning many approaches to sampling, identification, and inference \citep{kac1985enigmas,hinton2002training,wasserman2004all,imbens2004nonparametric, bishop2006pattern,durrett2019probability}. Independence is also a practical objective in engineering \citep{benner2005synthetic,goodfellow2016deep,sastry2019escherichia}. Yet independence is often assumed rather than empirically established, even though unrecognized dependence can invalidate conclusions drawn from these approaches \citep{kruskal1988miracles,berk2003statistical,bengio2004unbiased}. Determining when observations can justify independence is therefore a fundamental problem. This raises several questions: How broadly can independence be learned consistently? Under which assumptions is such learning possible? Under which is it impossible? Where does the boundary between the two lie?

For time-series observations, the property of interest in practice is independence of the two entire processes \citep{bach2004learning,tank2015bayesian}. To learn whether they are independent, we need a test whose chance of giving the wrong answer approaches zero as more observations become available \citep{gretton2010consistent}. However, existing theoretical analyses often address more tractable surrogate conditions, such as zero cross-correlations or pairwise independence at finitely many time lags, or even at all lags \citep{hong1996testing,shao2009generalized,besserve2013statistical,chwialkowski2014kernel}. These conditions are generally much weaker than full-process independence. Consequently, we can learn a surrogate target while still failing to detect process-level dependence \citep{shao2009generalized,besserve2013statistical,chwialkowski2014kernel}.

A natural way to bridge this gap is to impose model assumptions under which tractable surrogate conditions characterize full-process independence. Examples include jointly Gaussian models, specified graphical dependence structures, and parametric models \citep{bach2004learning,zhang2008kernel,shao2009generalized,besserve2013statistical}. Any conclusion of independence drawn within such a model, however, is limited to the assumed class and may fail if the true generating mechanism falls outside it.

An alternative approach is to restrict dependence across observation times, for example through i.i.d. sampling or mixing assumptions \citep{gretton2008kernel,gretton2010consistent,chwialkowski2014kernel,chwialkowski2014wild,rubenstein2016kernel,davis2018applications,pfister2018kernel,chu2023distance}. In engineering applications, however, signals are often recorded during system operation, while independently repeating the experiment under the same conditions may be impractical \citep{ljung1987system,tank2015bayesian,simchowitz2018learning,fattahi2019learning}. This motivates studying one increasingly long recording of the two signals. The observation scheme alone does not justify assuming independence across time or a particular form or rate of decay of temporal dependence. Moreover, guarantees under specific temporal assumptions do not by themselves identify the general boundary deciding learnability.

We therefore seek general statistical conditions governing this boundary, without specifying the generating mechanisms or probability distributions of the signals. We adopt the broad nonparametric framework of jointly stationary ergodic processes. Joint stationarity means that the joint law is invariant under time shifts, while ergodicity ensures that time averages of integrable observables converge almost surely to their expectations under that law. These assumptions accommodate a broad range of signal-generating mechanisms without prescribing a parametric model or a quantitative rate of temporal decorrelation.

Nevertheless, stationarity and ergodicity alone do not suffice for consistent testing of independence. \citet{ryabko2017independence} showed that no test can keep the false-positive probability below a prescribed level at every sample size while almost surely detecting every dependent law eventually. Pointwise weak consistency is also impossible in general, even when both processes are binary-valued \citep{yao2026independence}.

Two possible responses are to impose additional assumptions on the processes or to relax the performance required of the test. \citet{fernandes2010entropy} pursue the former direction through mixing and regularity conditions. \citet{harris2020shift} and \citet{yuan2024rigorous} provide finite-sample false-positive control under stationarity of one coordinate, without establishing a general guarantee of consistent detection. These two directions motivate our question: which additional assumptions permit a weakly consistent test of full-process independence?

We formulate our assumption through the shared nontrivial \emph{Koopman eigenvalues}, which represent common persistent oscillatory frequencies of the coordinate processes. Within the jointly stationary ergodic class, the absence of such shared eigenvalues is necessary for independence, although it does not itself imply independence. This concept enables us to give a class-level boundary for independence testability:
\begin{itemize}
    \item The \emph{absence of shared nontrivial Koopman eigenvalues} suffices for weakly consistent testing of full-process independence, with finite-sample false-positive control and no quantitative mixing assumptions. This regime includes cases in which either coordinate is weakly mixing, while also allowing both coordinates to be nonmixing. We construct an \emph{explicit deterministic testing algorithm} for this regime by extending the shift-test framework of \citet{harris2020shift} and \citet{yuan2024rigorous}, combining their finite-sample rank calibration with a growing family of finite-window dependence witnesses and a data-adaptive threshold.

    \item No possibly randomized test can have vanishing false-positive probability under every independent jointly totally ergodic law while consistently detecting every dependent jointly totally ergodic law with \emph{a shared nontrivial Koopman eigenvalue}. Indeed, for a broad class of independent laws, any such test falsely declares dependence with probability arbitrarily close to one, even at arbitrarily large sample sizes. The obstruction persists even when the shared point spectrum is exactly the cyclic group generated by any single prescribed irrational eigenfrequency.
\end{itemize}

The remainder of this paper is organized as follows. Section~\ref{sec:preliminaries} introduces needed concepts. Section~\ref{sec:main_results} presents the main results. Section~\ref{sec:positive_regime} proves the positive regime. Section~\ref{sec:negative_regime} proves the negatice regime. Finally, concluding remarks and discussions will be given in Section~\ref{sec:discussion}.

\section{Preliminaries}
\label{sec:preliminaries}

\paragraph{Statistical testing.} For integers \(a\leq b\), we call \([a,b]\) a window of length \(b-a+1\) and \(Z_{a:b}\) the corresponding observed block. The aim of a statistical testing problem is to use an observed record \(Z_{1:n}\) to decide whether to reject the null hypothesis \(\mathcal{H}_0\) in favor of the alternative hypothesis \(\mathcal{H}_1\). A deterministic test \(\varphi_n\) is a measurable function of \(Z_{1:n}\) taking values in \(\{0,1\}\), where one means rejection and zero means non-rejection. A randomized test may take values in \([0,1]\): after observing \(Z_{1:n}\), it rejects with probability \(\varphi_n(Z_{1:n})\), using additional randomness independent of the observations. Under a law \(P\), the unconditional rejection probability is \(\mathbb E_P[\varphi_n]\). Rejecting a true null is a Type~I error; failing to reject under an alternative is a Type~II error. A test has finite-sample level \(\alpha_0\) if its Type~I error is at most \(\alpha_0\) for every null law. The power of a test is the rejection probability under an alternative. A test sequence \((\varphi_n)\) is pointwise weakly consistent if \(\mathbb{E}_P[\varphi_n]\to0, P\in\mathcal H_0\) and \(\mathbb{E}_P[\varphi_n]\to1, P\in\mathcal H_1.\)

\paragraph{Ergodic theory.} We consider processes indexed by \(\mathbb Z\), taking values in standard Borel spaces. Their path spaces carry the product \(\sigma\)-algebra. Write \(P\) for a process law and \(T\) for the left shift \((Tx)_t=x_{t+1}\). The law \(P\) is called stationary if \(P\circ T^{-1}=P\). A stationary process law together with its shift defines a probability-preserving system: applying \(T\) preserves the probability measure \(P\). A stationary law is called ergodic if every measurable event \(A\) satisfying \(P(T^{-1}A\triangle A)=0\) has probability zero or one, where \(\triangle\) denotes symmetric difference. More generally, a measurable observable \(X_0\) on a probability-preserving system generates a stationary process through \(X_t=X_0\circ T^t\). Throughout, transformations are invertible with measurable inverses, as are two-sided shifts. Let \(L^2(P)\) denote the Hilbert space with inner product \(\langle f,g\rangle=\mathbb E_P[f\overline g]\). The Koopman operator is defined by \(Uf=f\circ T\). It is unitary for an invertible probability-preserving transformation, so every eigenvalue satisfies \(|\lambda|=1\) \citep[Proposition~7.18(b)]{eisner2015operator}. The point spectrum is the set of eigenvalues \citep[Appendix~C.9, p.~519]{eisner2015operator}. Ergodicity is equivalent to the \(1\)-eigenspace consisting only of constants \citep[Proposition~7.15]{eisner2015operator}. The Kronecker factor \(\mathcal K\) is the \(P\)-completed sub-\(\sigma\)-field generated by all Koopman eigenfunctions, with associated restricted measure and transformation.

\paragraph{Hierarchy under stationarity.} A stationary law is weakly mixing if its Cartesian square is ergodic under \(P\otimes P\) \citep[Theorem~9.23]{eisner2015operator}. Equivalently, its Koopman operator has no nonconstant eigenfunctions \citep[Theorem~9.25]{eisner2015operator}, or its Kronecker factor is trivial \citep[Proposition~17.16]{eisner2015operator}. It is totally ergodic (TE) if \(T^m\) is ergodic for every integer \(m\geq1\). For an ergodic system, total ergodicity is equivalent to the absence of eigenvalues other than \(1\) that are roots of unity \citep[Chapter~18, Exercise~17, p.~402]{eisner2015operator}. Thus, weak mixing \(\Rightarrow\) TE \(\Rightarrow\) ergodicity. Under stationarity, weak mixing is strictly weaker than commonly used assumptions such as \(\alpha\)-, \(\beta\)- and \(\phi\)-mixing \citep{Chacon1969,bradley2005mixing}. 

\paragraph{Joint law and product law.} For a joint process law \(P\), write \(P_X,P_Y\) for its marginal process laws, \(T_X,T_Y\) for their shifts, and \(U_X,U_Y\) for their Koopman operators. Let \(\mathcal S\) denote the class of joint laws invariant under the simultaneous shift \(T=T_X\times T_Y\), and let \(\mathcal E\subseteq\mathcal S\) denote its ergodic subclass. Their point spectra are denoted by \(\Sigma_X(P)\) and \(\Sigma_Y(P)\). The product measure \(Q=P_X\otimes P_Y\) defines the independent product system, whose Koopman operator is \(U_X\otimes U_Y\). Note that joint ergodicity implies marginal ergodicity \citep[Proposition~12.11]{eisner2015operator}. Under marginal ergodicity, the product system is ergodic exactly when \(\Sigma_X(P)\cap\Sigma_Y(P)=\{1\}\), so shared nontrivial eigenvalues obstruct product ergodicity \citep[Theorem~12(2)]{Quas2012}. Let \(\mathcal E_{\mathrm{TE}}\) denote the class of jointly totally ergodic laws.

\paragraph{Geometry and topology.} Write \(\mathbb S^1 \cong \mathbb{R}/\mathbb{Z}\) for the unit circle and \(\mathbb T= \mathbb S^1 \times \mathbb S^1\) for the torus, carrying their usual topology and Borel \(\sigma\)-algebra. We equip the space of stationary laws with the weak topology. A subset of a topological space is nowhere dense if its closure has empty interior, and meagre if it is a countable union of nowhere dense sets. A set is residual if its complement is meagre. In a Baire space, every residual set is dense \citep[Lemma~48.1, p.~296]{Munkres2000}. 

\section{Main Results}
\label{sec:main_results}

For simplicity of exposition, we consider binary-valued processes. The positive result extends directly to standard Borel observation spaces, while the impossibility of weakly consistent testing already arises in the binary case.

Suppose we observe a single record \(Z_{1:n}\) from a binary pair process \(Z_t=(X_t,Y_t)\in\{0,1\}^2\), \(t\in\mathbb Z\). Its joint law \(P\), assumed stationary and ergodic, is defined on \(\Omega=(\{0,1\}^2)^{\mathbb Z}\), equipped with the product \(\sigma\)-algebra. Let \(\mathcal E\) denote the class of all such laws. For a subclass \(\mathcal C\subseteq\mathcal E\), specified in each result, we consider
\begin{equation}
\mathcal H_0:=\{P\in\mathcal C:P=P_X\otimes P_Y\},
\qquad
\mathcal H_1:=\mathcal C\setminus\mathcal H_0.
\end{equation}
The null requires independence of the two entire coordinate processes. Our objective is to determine which admissible classes \(\mathcal C\) permit consistent testability of process-level independence, meaning the existence of a pointwise weakly consistent test sequence.

To describe the class-level boundary, we introduce the Kronecker-nonresonant (KNR) and Kronecker-resonant (KR) classes:
\begin{equation}
\begin{aligned}
\mathcal C_{\mathrm{KNR}}
&:=\{P\in\mathcal E:\Sigma_X(P)\cap\Sigma_Y(P)=\{1\}\},\\
\mathcal C_{\mathrm{KR}}
&:=\{P\in\mathcal E:\Sigma_X(P)\cap\Sigma_Y(P)\neq\{1\}\}.
\end{aligned}
\end{equation}

Our first result shows that excluding shared nontrivial eigenvalues suffices for consistent testing, without imposing quantitative mixing assumptions.

\begin{theorem}[Consistent testability under KNR]
\label{thm:knr-testability}
For every \(\alpha_0\in(0,1)\), there exists a deterministic test sequence \((\varphi_n)\) that is pointwise weakly consistent on \(\mathcal C_{\mathrm{KNR}}\). Moreover, every test in this sequence has finite-sample level \(\alpha_0\).
\end{theorem}
Section~\ref{sec:positive_regime} constructs the test and proves these guarantees. Weak mixing of either marginal implies KNR because that marginal has no eigenvalues other than \(1\). Theorem~\ref{thm:knr-testability} therefore includes this familiar sufficient condition, but its scope is much broader.

The consistency argument for the proposed test relies on ergodicity of the product system under \(Q=P_X\otimes P_Y\), which KNR guarantees. Shared nontrivial eigenvalues destroy this property. Surprisingly, the resulting obstruction extends beyond the proposed test. Fix any \(\alpha\notin\mathbb Q\), put \(\lambda=e^{2\pi i\alpha}\), and write \(\langle\lambda\rangle:=\{\lambda^k:k\in\mathbb Z\}\). Define
\[
\mathcal C_\lambda:=\bigl\{P\in\mathcal E_{\mathrm{TE}}:\Sigma_X(P)\cap\Sigma_Y(P)=\langle\lambda\rangle\bigr\}.
\]
Integer powers are unavoidable (Lemma~\ref{app:eigenvalue-powers}). Every law in \(\mathcal C_\lambda\) is dependent (Lemma~\ref{app:s5:projection}). Denote the null as \(\mathcal N\) when \(\mathcal C=\mathcal E\). The following negative result holds for every prescribed \(\lambda\) above, even if the test knows it.

\begin{theorem}[Impossibility at a prescribed exact shared spectrum]
\label{thm:te-impossibility}
Fix any \(\alpha\notin\mathbb Q\) and \(\lambda=e^{2\pi i\alpha}\). Let \((\varphi_n)\) be any possibly randomized test sequence whose power tends to one under every alternative in \(\mathcal C_\lambda\). Then the set
\[\mathcal B_\varphi
:=
\left\{
Q\in\mathcal{N}\cap\mathcal E_{\mathrm{TE}}:
\limsup_{n\to\infty}\mathbb E_Q[\varphi_n]=1
\right\}
\]
is a dense \(G_\delta\) subset of \(\mathcal{N}\cap\mathcal E_{\mathrm{TE}}\) under the relative weak topology. Consequently, no possibly randomized test sequence is pointwise weakly consistent for distinguishing \(\mathcal{N}\cap\mathcal E_{\mathrm{TE}}\) from \(\mathcal C_\lambda\), and hence neither from \(\mathcal C_{\mathrm{KR}}\cap\mathcal E_{\mathrm{TE}}\).
\end{theorem}

Section~\ref{sec:negative_regime} proves this negative regime. The theorem identifies a class-level KR obstruction: detecting every alternative whose shared point spectrum is exactly \(\langle\lambda\rangle\) forces false-positive probabilities arbitrarily close to one on a residual set of independent TE laws. TE shows that this failure persists even after finite periodicity is excluded. This is much stronger than the result of \citet{yao2026independence}.

Figure~\ref{fig_class} summarizes the resulting class-level picture. 

\begin{figure}[h]
  \centering
  \includegraphics[width=\columnwidth]{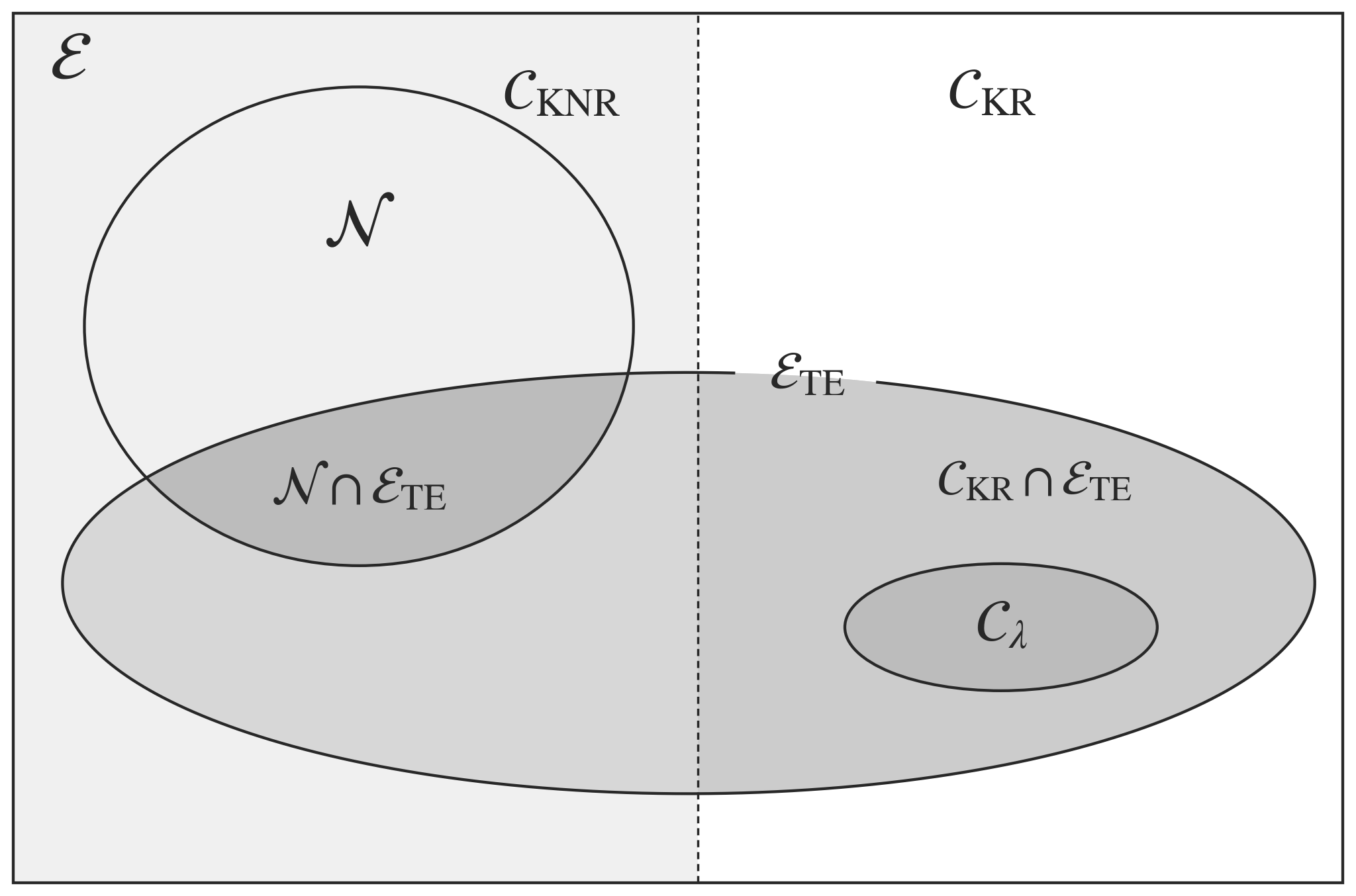}
  \caption{Class-level Kronecker boundary within \(\mathcal E\). \(\mathcal C_{\mathrm{KNR}}\) permits weakly consistent independence testing, whereas asymptotic power one throughout \(\mathcal C_\lambda\), for a prescribed irrational eigenfrequency \(\alpha\), forces false-positive probabilities to have limsup one on a residual subset of \(\mathcal N\cap\mathcal E_{\mathrm{TE}}\). No relative size, connectedness, or topological separation is implied by the shapes and placement of the regions.}
  \label{fig_class}
\end{figure}

\section{Learning Without Shared Eigenvalues}
\label{sec:positive_regime}

In this section, we prove Theorem~\ref{thm:knr-testability}. We develop the construction step by step, clarifying what existing methods provide, why our additional ingredients are needed, and where KNR enters.

\paragraph{Capturing dependence by a single score.} Our goal is to summarize full-process dependence in a single score. Since the joint law \(P\) is determined by its finite-dimensional distributions, any dependence is witnessed by a pair of events, each involving finitely many observations from its respective process. As the witnessing pair is unknown, we construct a score that combines all such potential witnesses. For finite-valued observations\footnote{For real-valued observations, it also suffices to consider a countable collection of event pairs.}, the collection of finite-window event pairs is countable, denoted as \((A_r, B_r), r \geq 1\). It can be ordered such that \((A_r, B_r)\) is supported on time indices in \(\{-r, \ldots, r\}\) (see Lemma~\ref{app:s4:enumeration}). Moreover, two events \(A_r, B_r\) are dependent exactly when their binary indicators \(F_r:=\mathbf 1_{A_r}, G_r:=\mathbf 1_{B_r}\) have nonzero covariance \(c_r:= \operatorname{Cov}_P(F_r,G_r)\). We combine these covariances into a single population score using exponentially decaying positive weights: \(\Delta_P:=\sum_{r\geq1}2^{-r}|c_r|\). It is obvious that \(\Delta_P = 0 \iff P \in \mathcal{H}_0\). To estimate this quantity, we compute the empirical covariance \(\hat{c}_{r,N}\) and the empirical score $V_N$ for finite samples. Here \(N\) controls how much of the record is used for each estimate; it grows with the total record length \(n\). Under mere joint ergodicity, \( V_N \to  \Delta_P, P-\text{a.s.}, P\in\mathcal E\) (see Proposition~\ref{app:s4:original-limit}).

Finite records can yield nonzero empirical covariances even under process independence, so a positive estimated score can trigger false alarms. Ergodicity ensures that these errors vanish but provides no convergence rate; a threshold that shrinks faster than the scores under independence may therefore keep producing false alarms as the record grows. Calibrating the threshold requires a reference for score magnitudes under independence at the same record length. Ideally, this would come from repeated records, but we observe only one trajectory from a jointly stationary process whose independence is unknown. We therefore reuse windows from different parts of it. For their scores to be comparable, window distributions must not depend on their starting positions; otherwise, score differences could reflect changes within either process rather than dependence between them.

\paragraph{Shift-based scores.} To turn these reused windows into a reference, we recompute the empirical dependence score on pairs of windows, one from each process. We keep the original first window fixed so that only the choice of the second window changes between comparisons. Pairing it with equally long windows starting at different positions in the second record amounts to shifting only the second signal. This will lead to the definition of shifted covariance
\begin{equation}
    c_r(s) := \operatorname{Cov}_P(F_r,G_r\circ T^s), \qquad s \in \mathbb{Z}
\end{equation}
and its empirical version
\begin{equation}
\begin{aligned}
\hat{c}_{r,N}&(s) :=
 \frac{1}{|\mathcal{I}_N|}\sum_{t\in \mathcal{I}_N}
 (F_r\circ T^t)(G_r\circ T^{t+s})
 \\&-
 \left(\frac{1}{|\mathcal{I}_N|}\sum_{t\in \mathcal{I}_N}F_r\circ T^t\right)
 \left(\frac{1}{|\mathcal{I}_N|}\sum_{t\in \mathcal{I}_N}G_r\circ T^{t+s}\right),
\end{aligned}
\end{equation}
where \(\mathcal{I}_N\) is the common averaging window used for all shifts specified below. Similarly, we can define the shifted score
\begin{equation}
 \Delta_P(s):=\sum_{r\geq1}2^{-r}|c_r(s)|
\end{equation}
and its empirical version
\begin{equation}
    V_N(s) := \sum_{r = 1}^{K_N}2^{-r}|\hat{c}_{r,N}(s)|, \qquad s \in \mathcal{S}_N.
\end{equation}
Here \(\mathcal{S}_N\) is the set of relative
shifts used to construct the reference, and \(K_N\), the feature budget, is the number of finite-window event pairs included in the empirical score, chosen so that \(1 \leq K_N \leq N\) and \(K_N \to \infty\) as \(N \to \infty\). We choose $\mathcal{I}_N$ so that all
observations required for these comparisons lie within the
available record. By stationarity, shifting the second window does not change its distribution. Under process independence, each shifted second window is also independent of the fixed first window. Thus, the joint distribution of the paired windows is unchanged by the shift, and consequently \(V_N(s)\) has the same distribution as \(V_N(0)\) for every fixed \(N\). Moreover, process independence implies \(c_r(s)=0\) for every \(r,s\), so \(\Delta_P(s)=0\) for every shift \(s\).

\paragraph{Rank-based calibration.} We now assess whether the original empirical score \(V_N(0)\) is unusually large relative to the shifted scores \(V_N(s)\). Specifically, we count the fraction of shifts whose scores are at least as large:
\begin{equation}
    p_N := \frac{1}{2N+1} \sum_{s \in \mathcal{S}_N} \mathbf{1}\{V_N(s) \ge V_N(0)\}.
    \label{eq:rank-def}
\end{equation}
A small \(p_N\) means that few shifted alignments produce a score as large as the original alignment. A shift-rank inequality from \citet{harris2020shift} and \citet{yuan2024rigorous} shows that, under process independence, the original score rarely ranks unusually high among the shifted scores, i.e., at every \(N\), 
\begin{equation}
 P(p_N\leq u)\leq2u,\qquad 0<u\leq\tfrac12,
 \quad P\in\mathcal H_0.
 \label{knr:eq:rank-calibration}
\end{equation}
The shift-test framework uses this inequality to control finite-sample Type-I error: with prescribed \(\alpha_0 \in (0,1)\), rejecting when \(p_N\leq\alpha_0/2\) gives a false-positive probability at most \(\alpha_0\).

Neither \citet{harris2020shift} nor \citet{yuan2024rigorous} establishes that both error probabilities tend to zero over the class considered here. A natural attempt is to choose a deterministic rejection threshold \(u_N\downarrow0\) and declare dependence whenever \(p_N\leq u_N\). The shift-rank inequality then makes the false-positive probability tend to zero. Nevertheless, this alone does not ensure even weak consistency under dependence. For a dependent law $P$, weak consistency requires the probability of declaring dependence $P(p_N\leq u_N)$ to tend to one. The shift-rank inequality applies only under independence and provides no such guarantee. Moreover, even proving \(p_N \to 0\) in probability would not suffice: the rejection threshold \(u_N\) also tends to zero, and \(p_N\) may remain above it if \(u_N\) decreases too quickly (see Appendix~\ref{app:s4:shrinking-threshold} for a counterexample). 

\paragraph{Adaptive threshold.} Convergence of the rank to zero alone does not justify a predetermined shrinking cutoff. We therefore use a data-adaptive rejection threshold. Our adaptive cut-off is
\begin{equation}
    u_N=\min\{\alpha_0/2,V_N(0)/4\} \label{knr:eq:test}
\end{equation}
(see Appendix~\ref{app:s4:adaptive-threshold} for its validity). We therefore define the test by
\begin{equation}
    \varphi_n=\mathbf 1\{p_N\leq u_N\}.
\end{equation}
Although its threshold is data-adaptive, \(\varphi_n\) is a deterministic test: each observed record uniquely determines its decision. We now choose \(N = \lfloor n/6\rfloor, \mathcal S_N=\{-N,\ldots,N\}\) and \(\mathcal I_N=\{2N+1,\ldots,4N\}\), whose validity will be shown in Lemma~\ref{app:s4:record}. Algorithm~\ref{knr:alg:test} summarizes the construction.

\begin{algorithm}[t]
 \caption{Score-Adaptive Shift-Rank Test}
 \label{knr:alg:test}
 \begin{algorithmic}[1]
  \STATE \textbf{Input:} $Z_{1:n}$, level $\alpha_0\in(0,1)$,
  and budget $(K_N)$.
  \STATE Set $N\leftarrow\lfloor n/6\rfloor$;
  if $N=0$, return $0$.
  \STATE Set $\mathcal I_N\leftarrow\{2N+1,\ldots,4N\}$
  and $\mathcal S_N\leftarrow\{-N,\ldots,N\}$.
  \STATE Compute the empirical covariances
  $\hat c_{r,N}(s)$ and scores $V_N(s)$ defined above,
  for $r\leq K_N$ and $s\in\mathcal S_N$.
  \STATE Compute $p_N$ by \eqref{eq:rank-def}
  and $u_N$ by \eqref{knr:eq:test}.
  \STATE \textbf{return} $\mathbf 1\{p_N\leq u_N\}$.
 \end{algorithmic}
\end{algorithm}

At this stage, the construction provides finite-sample Type-I error control and vanishing Type-I error under joint stationary ergodicity without invoking KNR. The difficult remaining step is to prove that the Type-II error vanishes—that is, the rejection probability tends to one under dependence. As anticipated in our construction, it suffices to show \(p_N \to 0\) in probability. A direct analysis of \(p_N\) asks which shifted scores \(V_N(s)\) exceed the original score \(V_N(0)\); convergence of individual scores alone does not determine this fraction. This motivates a different approach. 

\paragraph{Passing to average score.} Under dependence, \(V_N(0)\) approaches a positive value \(\Delta_P(0) > 0\). Every \(V_N(s)\) counted by \(p_N\) has a score at least as large as \(V_N(0)\). By nonnegativity of the scores and definition of \(p_N\), Eq.~\eqref{eq:rank-def}, the average shifted score bounds the rank fraction:
\begin{equation}
 S_N:=\frac1{2N+1}\sum_{s\in\mathcal S_N}V_N(s),
 \quad p_N V_N(0)\leq S_N.
 \label{knr:eq:rank-from-average}
\end{equation}
By Slutsky’s theorem, it suffices to show \(S_N \to 0\) in probability. By Markov's inequality, it suffices to show \(\mathbb{E}_P S_N \to 0\) (see Proposition~\ref{app:s4:score-vanishes} for a detailed proof).

\paragraph{Passing to average squared covariance.} Linearity of expectation separates the contributions of individual event pairs, while Cauchy–Schwarz's inequality bounds each contribution in terms of its averaged expected squared empirical covariance:
\begin{equation}
\begin{aligned}
\mathbb E_PS_N
& =
\sum_{r=1}^{K_N}2^{-r}
\frac1{2N+1}\sum_{s\in\mathcal S_N}\mathbb E_P|\hat c_{r,N}(s)|\\
& \leq \sum_{r=1}^{K_N}2^{-r} \sqrt{\frac1{2N+1}\sum_{s\in\mathcal S_N}
\mathbb E_P|\hat c_{r,N}(s)|^2}\\
\end{aligned}
\end{equation}
A truncation argument (see Appendix~\ref{app:s4:covariance-truncation})
shows that it suffices to prove
\begin{equation}
    \frac1{2N+1}\sum_{s\in\mathcal S_N}
\mathbb E_P|\hat c_{r,N}(s)|^2 \to 0, \quad \forall r \geq 1
\label{eq:averaged_squared_empirical_covariance_vanish}
\end{equation}

\paragraph{Passing to long average.} Fix \(r\), and center the indicators:
\begin{equation}
f=F_r-\mathbb E_PF_r, \quad g=G_r-\mathbb E_PG_r
\end{equation}
Define the average shift product of the centered indicators, called the ``long average'' for brevity:
\begin{equation}
A_N(s):=\frac1{|\mathcal I_N|}
\sum_{t\in\mathcal I_N}f(T^tX)g(T^{t+s}Y).
\end{equation}
It can be shown that replacing \(\hat c_{r,N}(s)\) by \(A_N(s)\) does not change the limiting averaged second moment (see Lemma~\ref{app:s4:centering-lemma} for a proof).

\paragraph{Passing to short average.} Now divide the long estimation interval into many short blocks of length \(H\). Define the average of the same products over a short block of length \(H\), called the ``short average'' for brevity:
\begin{equation}
B_H(s):=\frac1H\sum_{t=0}^{H-1}f(T^tX)g(T^{t+s}Y).
\end{equation}
By joint stationarity and Jensen’s inequality, the long average \(A_N(s)\) has no greater mean-square size than the short average \(B_H(s)\), apart from a boundary contribution that vanishes as the record grows (see Lemma~\ref{app:s4:block-lemma} for a proof). This yields
\begin{equation}
\begin{aligned}
&\limsup_{N\to\infty}\frac1{2N+1}
 \sum_s\mathbb E_P|\hat c_{r,N}(s)|^2\\
&=\limsup_{N\to\infty}\frac1{2N+1}
 \sum_s\mathbb E_P|A_N(s)|^2\\
&\le \lim_{N\to\infty}\frac1{2N+1}
 \sum_s\mathbb E_P|B_H(s)|^2.
\end{aligned}
\end{equation}

\paragraph{Passing to Koopman operator.} Keep \(H\) fixed, and write \(f_t=f(T_X^tX)\), \(g_t=g(T_Y^tY)\). Expanding the square gives
\begin{equation}
\begin{aligned}
&\frac1{2N+1}\sum_s\mathbb E_P|B_H(s)|^2
 \\
&=
\frac1{H^2}\sum_{t,u=0}^{H-1}
\mathbb E_P\!\left[
f_tf_u\,
\left(\frac1{2N+1}\sum_s g_{t+s}g_{u+s}\right)
\right].
\end{aligned}
\end{equation}
For each fixed pair \(t,u\), the term in parentheses converges to a constant \(\mathbb{E}_P[g_tg_u]\) almost surely by the ergodic theorem. Bounded convergence therefore yields
\begin{equation}
\begin{aligned}
 \lim_{N\to\infty}&\frac1{2N+1}\sum_s\mathbb E_P|B_H(s)|^2
\\
& =
\frac1{H^2}\sum_{t,u=0}^{H-1}
\mathbb E_P[f_tf_u]\,
\mathbb E_P[g_tg_u] \\
& = \frac1{H^2} \sum_{t,u=0}^{H-1} \mathbb E_Q[f_tf_ug_tg_u] = \mathbb E_Q\left|\frac1H\sum_{t=0}^{H-1}f_tg_t\right|^2 \\
& = \left\|\frac1H\sum_{t=0}^{H-1}(U_X\otimes U_Y)^t(f\otimes g)
 \right\|_{L^2(Q)}^2,\\
\end{aligned}
\end{equation}
where \(Q:=P_X\otimes P_Y\) is the product law.

\paragraph{Passing to eigenvalues.} The mean ergodic theorem gives
\begin{equation}
\frac1H\sum_{t=0}^{H-1}(U_X\otimes U_Y)^t(f\otimes g)
\longrightarrow \Pi_{\ker(U_X\otimes U_Y-\mathrm{id})}(f\otimes g)
\end{equation}
in \(L^2(Q)\) as \(H \to \infty\), where \(\Pi_{\ker(U_X\otimes U_Y-\mathrm{id})}\) denotes the orthogonal projection onto the \(1\)-eigenspace of \(U_X\otimes U_Y\). 

Under KNR, the projection in our formula is simply the projection onto constants. Because \(f,g\) are centered,
\begin{equation}
\Pi_{\ker(U_X\otimes U_Y-\mathrm{id})}(f\otimes g)
=\mathbb E_Q[f(X)g(Y)]\,\mathbf1=0.
\end{equation}
First taking \(N\to\infty\), then \(H\to\infty\), proves the convergence in~\eqref{eq:averaged_squared_empirical_covariance_vanish}. This also establishes Theorem~\ref{thm:knr-testability}.

\section{Impossibility of Universal Learning}
\label{sec:negative_regime}
\label{knr:sec:impossibility}

In this section, we prove Theorem~\ref{thm:te-impossibility}. Throughout we prescribe an irrational eigenfrequency \(\alpha\) and the corresponding shared nontrivial Koopman eigenvalue $\lambda=e^{2\pi i\alpha} \in \mathbb{S}^1$. We start with a simple example that illustrates how shared eigenvalues obstruct consistent independence testing.

\subsection{A torus prototype}
\label{knr:sec:torus-prototype}

Let \(R_\alpha: \mathbb{S}^1 \to \mathbb{S}^1\) denote the deterministic rotation \(R_\alpha(u) = u + \alpha\) on the unit circle. Geometrically, multiplication by \(\lambda\) rotates a point on this circle by \(\alpha\) turns. The simplest model is therefore two oscillators whose phases undergo this same rotation at each observation. In order to make this intuition a binary process, we introduce the indicator function \(b(u) = \mathbf{1}_{[0, 1/2)}(u)\) and endow each copy of \(\mathbb{S}^1\) with the uniform measure \(m\), under which the initial phase $\Theta$ is a uniformly distributed random variable.

\paragraph{Phase locking and resonant dependence.}
We formalize this prototype as follows. For each fixed relative phase $\beta\in\mathbb S^1$ and an initial phase \(\Theta \sim m\), we define
\begin{equation}
  D_\beta^{\mathrm{rot}} : \bigl(A_t,B_t^{(\beta)}\bigr)= \bigl(b(\Theta+t\alpha), b(\Theta+\beta+t\alpha)\bigr), \, t\in\mathbb{Z}
\end{equation}
The phases rotate at the same speed with a fixed relative phase, so they are dependent. Proposition~\ref{app:s5:prototype} establishes that
this dependence survives binary coding: \(A\) and \(B^{(\beta)}\)
retain the common eigenvalue \(\lambda\), preserving their Kronecker
resonance, their marginal point spectra are both exactly the cyclic
group \(\langle\lambda\rangle\), and the joint law
\(D_\beta^{\mathrm{rot}}\) inherits total ergodicity, so
\(D_\beta^{\mathrm{rot}}\in\mathcal C_\lambda\) is a dependent law. By the assumption of Theorem~\ref{thm:te-impossibility}, the test sequence satisfies \(\mathbb E_{D_\beta^{\mathrm{rot}}}[\varphi_n] \to 1\).

\paragraph{Haar averaging to restore independence.}
To obtain a law with independent coordinate processes, we draw \(\beta\sim m\) independently of \(\Theta\) and consider the Haar-averaged law
\begin{equation}
 \bar D^{\mathrm{rot}} := \int_{\mathbb S^1}D_\beta^{\mathrm{rot}}\,dm(\beta).
 \label{knr:eq:prototype-mixture}
\end{equation}
In fact, \(\bar D^{\mathrm{rot}}\) has independent coordinate processes. To see why, for any measurable event \(E\) involving the pair of paths,
\begingroup
\small
\begin{equation}
\begin{aligned}
&\bar D^{\mathrm{rot}}(E)\\
&=\int_{\mathbb S^1}\int_{\mathbb S^1}
\mathbf 1_E\!\left(
  (b(\theta+t\alpha),b(\theta+\beta+t\alpha))_{t\in\mathbb Z}
\right)
\,dm(\theta)\,dm(\beta)\\
&=\int_{\mathbb S^1}\int_{\mathbb S^1}
\mathbf 1_E\!\left(
  (b(\theta+t\alpha),b(v+t\alpha))_{t\in\mathbb Z}
\right)
\,dm(v)\,dm(\theta).
\end{aligned}
\end{equation}
\endgroup
This is given by Fubini's theorem and the measure-preserving map  \((\Theta, \beta) \mapsto (\Theta, \Theta+\beta)\). The last expression is the product of the two marginal path laws evaluated at \(E\), proving independence. Figure~\ref{fig:haar-averaging} illustrates this construction.

\begin{figure}[ht]
  \centering
  \includegraphics[width=\columnwidth]{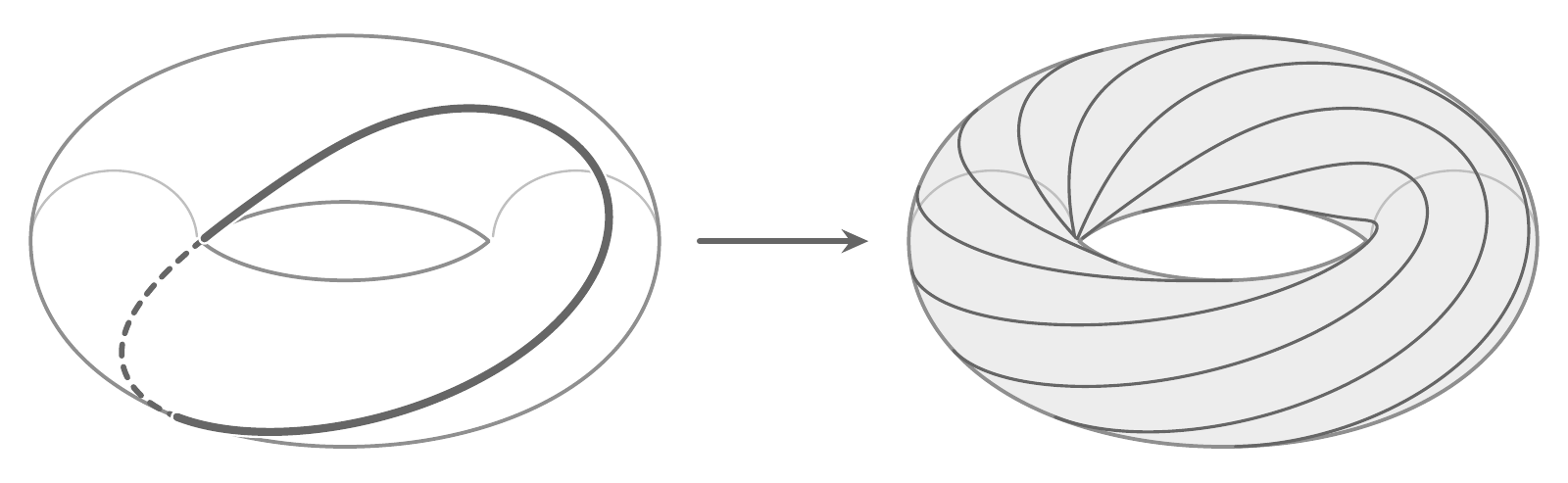}
  \caption{Haar averaging restores independence. Geometrically, each fixed relative phase \(\beta\) gives a uniform distribution on one diagonal circle of the phase torus \(\mathbb{T}\); Haar averaging over the relative phase yields the uniform distribution on the entire torus.}
  \label{fig:haar-averaging}
\end{figure}

Crucially, this averaging preserves the test's high rejection probability:
\begin{equation}
 \mathbb E_{\bar D^{\mathrm{rot}}}[\varphi_n]
   =\int_{\mathbb S^1}\mathbb E_{D_\beta^{\mathrm{rot}}}[\varphi_n]\,dm(\beta)
   \longrightarrow1.
 \label{knr:eq:prototype-mixture-power}
\end{equation}
However, the independent marginals still share the nontrivial
eigenvalue $\lambda$, so their product law $\bar D^{\mathrm{rot}}$
is not ergodic (see Proposition~\ref{app:s5:prototype}). We must therefore restore total ergodicity while
retaining the high rejection probability at a selected finite horizon.

\paragraph{Post-horizon frequency perturbation.}
We perturb the second frequency to \(\vartheta\) near \(\alpha\), and define
\begin{equation}
  Q_\vartheta^{\mathrm{rot}}: \bigl(b(\Theta+t\alpha),\ b(\Theta+\beta+t\vartheta)\bigr), \, t \in \mathbb{Z}.
 \label{knr:eq:prototype-split}
\end{equation}
The independent initial phases ensure process independence (see Appendix~\ref{app:prototype-kr-te}). When \(1,\alpha,\vartheta\) are \(\mathbb Q\)-linearly independent, the system is totally ergodic, so \(Q_\vartheta^{\mathrm{rot}}\in \mathcal N\cap\mathcal E_{\mathrm{TE}}\) (see Appendix~\ref{app:prototype-kr-te}). Such rationally independent frequencies occur arbitrarily close to \(\alpha\).

\begin{figure}[ht]
  \centering
  \includegraphics[width=\columnwidth]{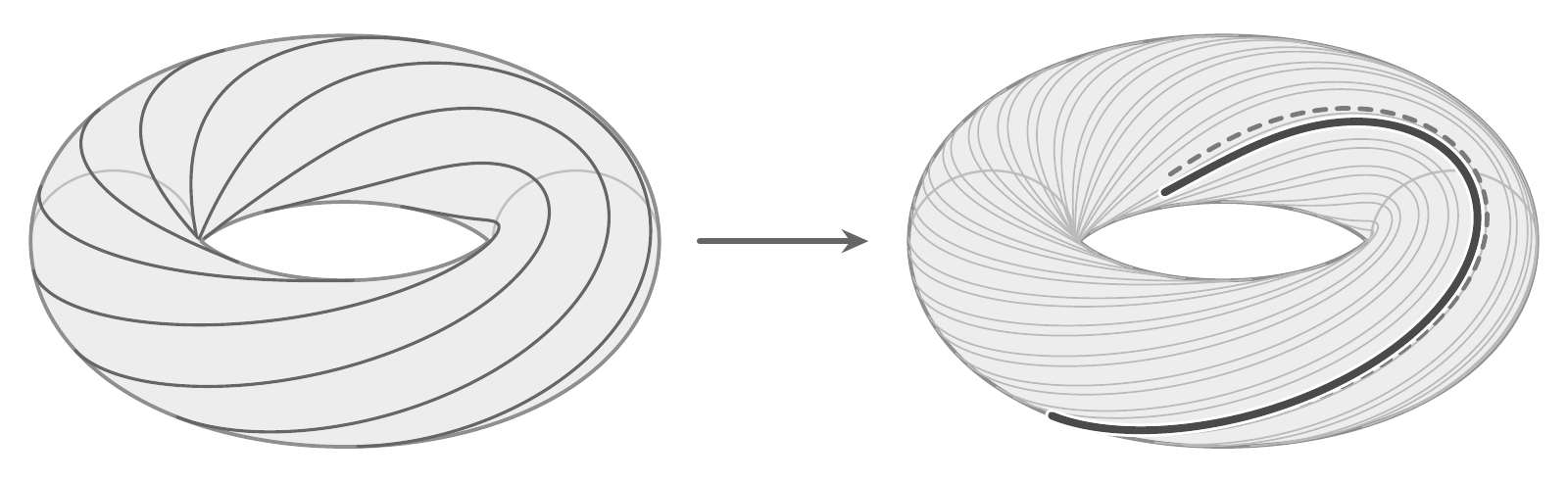}
  \caption{Frequency perturbation on the phase torus. The original (dashed) and perturbed (solid) traces remain close over a finite horizon, while the perturbed dynamics eventually explore the entire torus (pale winding, shown schematically). Shading denotes uniform independent initial phases in both graphs.}
  \label{fig:frequency-perturbation}
\end{figure}

As illustrated in Figure~\ref{fig:frequency-perturbation}, a small frequency perturbation changes the long-time dynamics while remaining almost invisible on a fixed finite horizon. The total variation distance is bounded by
\begin{equation}
 d_{\mathrm{TV}}\!\left(
      (Q_\vartheta^{\mathrm{rot}})_{1:n},
      (\bar D^{\mathrm{rot}})_{1:n}\right)
 \leq C n^2 d_{\mathbb S^{1}}(\vartheta,\alpha)
 \label{knr:eq:prototype-tv}
\end{equation}
(see Lemma~\ref{app:tv-bound}). By choosing horizon \(n\) \emph{before} selecting \(\vartheta\), we can make this distance arbitrarily small. This forces the valid null law \(Q_\vartheta^{\mathrm{rot}}\) to be falsely rejected with high probability:
\begin{equation}
 \begin{aligned}
 \mathbb E_{Q_\vartheta^{\mathrm{rot}}}[\varphi_n] \geq c,
 \end{aligned}
 \label{knr:eq:prototype-large-rejection}
\end{equation}
for any threshold \(c\in(0,1)\). Appendix~\ref{app:fixed-torus-null} constructs a fixed null law
\(Q_{\vartheta^\star}^{\mathrm{rot}}
\in\mathcal N\cap\mathcal E_{\mathrm{TE}}\) such that \(
  \limsup_{n\to\infty}
  \mathbb E_{Q_{\vartheta^\star}^{\mathrm{rot}}}[\varphi_n]=1.
\)

\subsection{Embedding the mechanism near an arbitrary null}
\label{knr:sec:prototype-generalization}

The torus prototype produces independent, totally ergodic laws that the test rejects with high probability at arbitrarily late horizons. We now reproduce this obstruction near an arbitrary \(P_0\in\mathcal N\cap\mathcal E_{\mathrm{TE}}\). Fix a relative neighborhood \(\mathcal O\) of \(P_0\), and choose a weakly open neighborhood \(\mathcal G\) of \(P_0\) in the space of stationary laws \(\mathcal{S}\) such that \(\mathcal G \cap(\mathcal N\cap\mathcal E_{\mathrm{TE}})\subseteq\mathcal O\).

Write \(P_0=\mu\otimes\nu\), let \((X^0,Y^0)\sim P_0\), and draw \(\Theta\sim m\) independently of the background process \((X^0,Y^0)\). For each fixed relative phase \(\beta\), use this \(\Theta\) to generate the torus signals \(A\) and \(B^{(\beta)}\) defined above, which we will combine with \((X^0,Y^0)\) to construct the nearby alternatives \(D_\beta\).

\paragraph{Rare marker.}
For \(p\in(0,1)\), let \(M^X\) and \(M^Y\) be independent i.i.d.\ Bernoulli(\(p\)) processes, independent of \((X^0,Y^0,\Theta)\). For each fixed \(\beta\in\mathbb S^1\), define \(D_\beta\) as
\begin{equation}
D_\beta: \quad
 \begin{aligned}
 X_t^{(\beta)}&=(1-M_t^X)X_t^0+M_t^X A_t\\
 Y_t^{(\beta)}&=(1-M_t^Y)Y_t^0+M_t^Y B_t^{(\beta)}
 \end{aligned}
 ,\quad t\in\mathbb{Z}.
 \label{eq:prototype-embedding}
\end{equation}

Each marker replaces a background bit with the corresponding rotation bit. We can choose \(p>0\) sufficiently small that \(D_\beta\in \mathcal G\) for every \(\beta\). Since proximity to \(P_0\) is ensured by the rarity of the markers, the same choice of \(p\) also keeps the subsequent averaged and frequency-perturbed laws in \(\mathcal G\) (see Proposition~\ref{app:s5:marker-eigenvalue}). However, since the prescribed rotation \(R_\alpha\) may share a nontrivial eigenvalue with \(P_0\), the product \(P_0 \times R_\alpha \) may fail to be totally ergodic. Without loss of generality (see Lemma~\ref{app:s5:mixing-reduction}), we may assume that \(P_0\) is an independent mixing law. Then \(P_0\times R_\alpha\) is totally ergodic with point spectrum exactly \(\langle\lambda\rangle\) (see Lemma~\ref{app:s5:mixing-spectrum}). Each \(D_\beta\) belongs to \(\mathcal C_\lambda\) (see Proposition~\ref{app:s5:marker-eigenvalue}). Hence the power assumption gives \(\mathbb E_{D_\beta}[\varphi_n]\to1\) for every \(\beta\).

\paragraph{Transferring the obstruction via a Markov kernel.}
We now view rare-marker replacement as a randomized procedure acting on a torus input process. On the window \(1:n\), this procedure defines a Markov kernel \(K_{p,n}\). The kernel depends on \(P_0,p,n\), but not on the frequencies \(\alpha,\vartheta\) generating the input bits (see Lemma~\ref{app:markov-kernel}). Define \(\bar D\) and \(Q_\vartheta\) as the output laws when the input laws are respectively \(\bar D^{\mathrm{rot}}\) and \(Q_\vartheta^{\mathrm{rot}}\). Since a Markov kernel cannot increase total variation distance (see Lemma~\ref{app:markov-kernel}), this will lead to
\begin{equation}
 \begin{aligned}
d_{\mathrm{TV}} &\!\left(
   (Q_\vartheta)_{1:n},\bar D_{1:n}
 \right)\\
 &\le
 d_{\mathrm{TV}}\!\left(
   (Q_\vartheta^{\mathrm{rot}})_{1:n},
   (\bar D^{\mathrm{rot}})_{1:n}
 \right)\\
 &\le
 n(n+1)d_{\mathbb S^{1}}(\vartheta,\alpha).
 \end{aligned}
 \label{eq:prototype-embedded-tv}
\end{equation}
The choice of \(p\) already ensures \(Q_\vartheta\in \mathcal G\).

\paragraph{High rejection neighbors.}
Note that \(\bar D\) is just the Haar-averaged \(D_\beta\) and \(\bar D=Q_\alpha\) (see Appendix~\ref{app:s5:phase-mixture}). By dominated convergence, with \(p\) fixed, for any \(k\ge1\) and \(c\in(0,1)\), first choose \(n\ge k\) such that \(\mathbb E_{\bar D}[\varphi_n]>c\). Then choose \(\vartheta\) sufficiently close to \(\alpha\) that \(\mathbb E_{Q_\vartheta}[\varphi_n]>c\), while avoiding the countable exceptional set of frequencies for which the full latent system fails to be totally ergodic (see Lemma~\ref{app:s5:split-te} and Proposition~\ref{app:s5:open-dense}). Then \(Q_\vartheta\in \mathcal G\cap(\mathcal N\cap\mathcal E_{\mathrm{TE}})
\subseteq\mathcal O\). Thus, every relative weak neighborhood \(\mathcal O\) of \(P_0\) in \(\mathcal N\cap\mathcal E_{\mathrm{TE}}\) contains a null law \(Q_\vartheta\) rejected with probability greater than \(c\) at an arbitrarily late horizon \(n\). 

\paragraph{Baire Diagonalization.}
The remaining task is to remove the dependence on the chosen horizon \(n\) and threshold \(c\). For \(k\ge1\) and \(c\in(0,1)\), let \(\mathcal U_{k,c}:=\bigcup_{n\ge k}
\left\{P\in\mathcal N\cap\mathcal E_{\mathrm{TE}}:
\mathbb E_P[\varphi_n]>c\right\}\). These sets are open in the relative weak topology on \(\mathcal N\cap\mathcal E_{\mathrm{TE}}\) by finite-window continuity and dense in that space by the preceding construction. Baire's theorem therefore implies that \(\bigcap_{j\ge2}\bigcap_{k\ge1}\mathcal U_{k,1-1/j}\) is a dense \(G_\delta\) set of fixed null laws satisfying \(\limsup_n\mathbb E_P[\varphi_n]=1\) (see Appendix~\ref{app:baire}).

\section{Conclusion and Discussion}
\label{sec:discussion}

Our research follows a tradition of foundational theory: we begin with a basic question, identify the structure underlying a common guarantee, and prove an obstruction beyond the scope of that guarantee \citep{BahadurSavage1956,XieGuo2000,SahaiMitter2006,ShahPeters2020,BousquetEtAl2021}. In our setting, the basic question is whether full-process independence can be tested consistently from a single increasingly long observation of two jointly stationary ergodic processes. 

We identify a class-level boundary of such learnability. The absence of shared nontrivial Koopman eigenvalues supports a concrete weakly consistent test. No single test can maintain false-positive control over every independent law while consistently detecting every dependent law with any prescribed irrational shared Koopman eigenfrequency.

A striking aspect of these results is that a purely spectral property of the marginal dynamics constrains the learnability of full-process independence. Independence concerns the entire collection of finite-dimensional joint distributions, so its connection to shared Koopman eigenvalues is far from apparent from the testing problem itself. In hindsight, earlier rotation-based constructions provide concrete intuition for the obstruction, which, however, do not by themselves identify the general structure \citep{ryabko2017independence,yuan2024rigorous,yao2026independence}.

In fact, the spectral structure appearing here is encoded by the Kronecker factor, which is also the first nontrivial level of the Host--Kra--Ziegler hierarchy \citep{HostKra2005,Ziegler2007} (see Appendix~\ref{app:kronecker}). Our results give it an explicit role in a statistical learning problem. It triggers further questions: Which restrictions on Kronecker-resonant alternatives restore a common consistency guarantee? Do higher Host--Kra--Ziegler factors similarly constrain the scope of guarantees for other statistical learning targets?


\bibliography{main}
\bibliographystyle{icml2026}

\newpage
\appendix
\onecolumn
\section{Proof details for Section 3}

\begin{lemma}[Integer powers of Koopman eigenvalues]
\label{app:eigenvalue-powers}
Let \((\Omega,P,T)\) be an ergodic probability-preserving system,
and let \(\Sigma(T)\) denote its Koopman point spectrum.
If \(\lambda\in\Sigma(T)\), then
\[
\langle\lambda\rangle
:=\{\lambda^k:k\in\mathbb Z\}\subseteq\Sigma(T).
\]
Consequently, for every jointly ergodic law \(P\),
\[
\lambda\in\Sigma_X(P)\cap\Sigma_Y(P)
\quad\Longrightarrow\quad
\langle\lambda\rangle
\subseteq\Sigma_X(P)\cap\Sigma_Y(P).
\]
\end{lemma}

\begin{proof}
Choose a nonzero eigenfunction \(f\in L^2(P)\) satisfying
\(f\circ T=\lambda f\).
Since \(|\lambda|=1\), the function \(|f|\) is invariant.
Ergodicity therefore implies that \(|f|=c\) almost surely
for some constant \(c>0\).
After normalization, we may assume that \(|f|=1\) almost surely.
For every \(k\in\mathbb Z\), the function \(f^k\) is thus
well-defined, nonzero, and belongs to \(L^2(P)\), with
\[
f^k\circ T=(f\circ T)^k=\lambda^k f^k.
\]
Hence \(\lambda^k\in\Sigma(T)\).
The final assertion follows by applying this argument to both
marginal systems, which inherit ergodicity from the joint law.
\end{proof}

\section{Proof details for Section 4}
\label{app:section4}

This appendix supplies the details of the score-adaptive shift-rank construction
in Section~4 and proves the positive testability result. We use the notation and empirical covariances of Section~4. Fix a
deterministic budget satisfying \(1\leq K_N\leq N\) and \(K_N\to\infty\),
and set
\begin{equation}
 N=\lfloor n/6\rfloor,\qquad
 \mathcal I_N=\{2N+1,\ldots,4N\},\qquad
 \mathcal S_N=\{-N,\ldots,N\}.
 \label{app:s4:windows}
\end{equation}
For \(n<6\), return \(\varphi_n=0\). For \(N\geq1\), write
\begin{align}
 V_N(s)&=\sum_{r=1}^{K_N}2^{-r}|\widehat c_{r,N}(s)|,
 &p_N&=\frac1{2N+1}\sum_{s=-N}^{N}\mathbf1\{V_N(s)\geq V_N(0)\},
 \label{app:s4:rank}\\
 u_N&=\min\{\alpha_0/2,V_N(0)/4\},
 &\varphi_n&=\mathbf1\{p_N\leq u_N\}.
 \label{app:s4:test}
\end{align}
Here \(\alpha_0\in(0,1)\). When proving stationarity in the shift index,
extend the same full-path formulas to all \(s\in\mathbb Z\); the test
itself uses only \(s\in\mathcal S_N\).

\subsection{Cylinder enumeration, completeness, and finite-record validity}
\label{app:s4:cylinders}

For a nonnegative integer \(m\), let \(\mathcal A_{X,m}\) be the collection
of all events determined by \(X_{-m:m}\), and define
\(\mathcal A_{Y,m}\) analogously. These collections are finite: an event
determined by \(X_{-m:m}\) is the inverse image of a subset of the finite
set \(\{0,1\}^{2m+1}\). Thus
\[
 \mathcal A_X=\bigcup_{m\geq0}\mathcal A_{X,m},\qquad
 \mathcal A_Y=\bigcup_{m\geq0}\mathcal A_{Y,m}
\]
are countable algebras generating the respective path \(\sigma\)-fields.
Every event depending on any finite set of time indices belongs to one of
these algebras, since its support is contained in some \([-m,m]\).

\begin{lemma}
\label{app:s4:enumeration}
There is a deterministic sequence \((A_r,B_r)_{r\geq1}\) that contains every
pair in \(\mathcal A_X\times\mathcal A_Y\) at least once and satisfies
\[
 A_r\in\mathcal A_{X,r},\qquad B_r\in\mathcal A_{Y,r}.
\]
\end{lemma}

\begin{proof}
Start with any enumeration \((C_j,D_j)_{j\geq1}\) of all cylinder pairs.
Let \(\rho_j\) be the least nonnegative integer such that both events in
this pair are determined by coordinates in \([-\rho_j,\rho_j]\).
At stage \(r\), choose the pair with smallest original index \(j\) that
has not previously been chosen and satisfies \(\rho_j\leq r\). If there
is no such pair, use the fixed trivial pair \((\Omega_X,\Omega_Y)\).
Every pair used at stage \(r\) has the required support.

To verify completeness, fix \(j\) and choose
\(r_0\geq\max_{1\leq i\leq j}\rho_i\). From stage \(r_0\) onward,
all unselected pairs among the first \(j\) original pairs are eligible.
As long as any remain, the earliest eligible unselected pair has original
index at most \(j\). Therefore all first \(j\) pairs have been selected
after at most \(j\) additional stages. In particular, the \(j\)th pair
eventually occurs.
\end{proof}

Fix such an enumeration, and put
\begin{equation}
 F_r=\mathbf1_{A_r},\qquad G_r=\mathbf1_{B_r},\qquad
 c_r(s)=\mathbb E_P[F_r(G_r\circ T^s)]-\mathbb E_PF_r\,\mathbb E_PG_r,
 \qquad \Delta_P=\sum_{r\geq1}2^{-r}|c_r(0)|.
 \label{app:s4:population}
\end{equation}
Stationarity makes \(\mathbb E_P[G_r\circ T^s]=\mathbb E_PG_r\), so
\(c_r(s)\) is the covariance in Section~4.

\begin{lemma}
\label{app:s4:completeness}
For every probability law \(P\) on the joint path space,
\(\Delta_P=0\) if and only if \(P=P_X\otimes P_Y\). Moreover,
\(0\leq\Delta_P\leq1/4\).
\end{lemma}

\begin{proof}
Independence gives \(c_r(0)=0\) for every \(r\). Conversely, if
\(\Delta_P=0\), positivity of the weights gives \(c_r(0)=0\) for every
\(r\). Completeness of the enumeration implies
\[
 P(A\cap B)=P(A)P(B),\qquad A\in\mathcal A_X,\ B\in\mathcal A_Y.
\]
Fix \(A\in\mathcal A_X\). The collection of \(B\in\mathcal F_Y\)
satisfying this identity is a \(\lambda\)-system: it contains the whole
space and is closed under complements and countable disjoint unions.
It contains the generating \(\pi\)-system \(\mathcal A_Y\), so the
\(\pi\)-\(\lambda\) theorem extends the identity to every
\(B\in\mathcal F_Y\). For each such \(B\), apply the same argument to
the class of \(A\in\mathcal F_X\) satisfying the identity. This proves
independence of \(\mathcal F_X\) and \(\mathcal F_Y\), equivalently
\(P=P_X\otimes P_Y\).

For indicators, Cauchy--Schwarz gives
\[
 |\operatorname{Cov}_P(F_r,G_r)|
 \leq\sqrt{\operatorname{Var}_P(F_r)\operatorname{Var}_P(G_r)}
 \leq\tfrac14.
\]
Since \(\sum_{r\geq1}2^{-r}=1\), the asserted bound follows.
\end{proof}

\begin{lemma}
\label{app:s4:record}
For \(r\leq K_N\), \(t\in\mathcal I_N\), and
\(s\in\mathcal S_N\), every coordinate needed to evaluate
\(F_r(T_X^tX)\) and \(G_r(T_Y^{t+s}Y)\) belongs to \(\{1,\ldots,n\}\).
\end{lemma}

\begin{proof}
The first event uses indices in \([t-r,t+r]\). As \(r\leq N\), these
indices lie in
\[
 [2N+1-N,\,4N+N]=[N+1,5N]\subseteq[1,n].
\]
The second event uses indices in \([t+s-r,t+s+r]\), which lie in
\[
 [2N+1-N-N,\,4N+N+N]=[1,6N]\subseteq[1,n].
\]
The last inclusion follows from \(6N\leq n\). The empirical means use
the same evaluations, so no additional indices are required. In particular,
the construction uses actual record entries and involves no circular wrapping
or extension of the observed data.
\end{proof}

\subsection{Almost-sure convergence of the score at the original alignment}
\label{app:s4:score-limit}

For a bounded observable \(h\), let
\(C_mh=m^{-1}\sum_{t=1}^m h\circ T^t, \quad D_Nh=\frac1{2N}\sum_{t=2N+1}^{4N}h\circ T^t.\) Under ergodicity,
\begin{equation}
 D_Nh
 =2C_{4N}h-C_{2N}h\longrightarrow\mathbb E_Ph
 \quad\text{almost surely and in }L^2(P).
 \label{app:s4:moving-window}
\end{equation}

Write the empirical covariance, on the full path space, as
\begin{equation}
 \widehat c_{r,N}(s)=
 \frac1{2N}\sum_{t\in\mathcal I_N}F_r(T_X^tX)G_r(T_Y^{t+s}Y)
 -\left(\frac1{2N}\sum_{t\in\mathcal I_N}F_r(T_X^tX)\right)
  \left(\frac1{2N}\sum_{t\in\mathcal I_N}G_r(T_Y^{t+s}Y)\right).
 \label{app:s4:empirical}
\end{equation}
The formula is meaningful for every integer \(s\); the test computes it only
for \(s\in\mathcal S_N\).

\begin{proposition}[Score convergence]
\label{app:s4:original-limit}
If \(P\) is stationary and ergodic, then, for every fixed integer \(s\),
\[
 V_N(s):=\sum_{r=1}^{K_N}2^{-r}|\widehat c_{r,N}(s)|
 \longrightarrow
 \Delta_P(s):=\sum_{r\geq1}2^{-r}|c_r(s)|
 \quad\text{almost surely and in }L^1(P).
\]
In particular, \(V_N(0)\to\Delta_P\).
\end{proposition}

\begin{proof}
For fixed \(r,s\), apply \eqref{app:s4:moving-window} to the three bounded
observables \(F_r\), \(G_r\circ T^s\), and
\(F_r(G_r\circ T^s)\). Equation~\eqref{app:s4:empirical} then gives
\(\widehat c_{r,N}(s)\to c_r(s)\) almost surely. A countable intersection
makes these conclusions simultaneous for all \(r\) and all fixed integer
\(s\).

The covariance under the empirical distribution of the \(2N\) observed
indicator pairs obeys the same bound as a population covariance:
\begin{equation}
 |\widehat c_{r,N}(s)|\leq\tfrac14,
 \qquad |c_r(s)|\leq\tfrac14,
 \qquad 0\leq V_N(s)\leq\tfrac14.
 \label{app:s4:bounded-scores}
\end{equation}
Indeed, empirical Cauchy--Schwarz bounds the covariance by the geometric
mean of the two empirical Bernoulli variances, each at most \(1/4\). Fix \(R\geq1\). For all sufficiently large \(N\), we have \(K_N\geq R\).
\begin{align}
 |V_N(s)-\Delta_P(s)|
 &\leq
 \sum_{r=1}^{R}2^{-r}
 \bigl|\widehat c_{r,N}(s)-c_r(s)\bigr|
 +\frac12\sum_{r>R}2^{-r}
 \notag\\
 &=
 \sum_{r=1}^{R}2^{-r}
 \bigl|\widehat c_{r,N}(s)-c_r(s)\bigr|
 +2^{-R-1}.
 \label{app:s4:original-score-tail}
\end{align}
First let \(N\to\infty\), and then \(R\to\infty\). This proves the
almost-sure limit; bounded convergence proves the \(L^1\) limit.
\end{proof}

\subsection{A counterexample for deterministic shrinking cutoffs}
\label{app:s4:shrinking-threshold}

The deterministic sequences
\[
 q_N=\frac{\lceil\sqrt{2N+1}\rceil}{2N+1},\qquad
 v_N=\frac1{2N+1}
\]
satisfy \(q_N\to0\) and \(v_N\downarrow0\), but
\(P\{q_N\leq v_N\}=0\) for every \(N\geq1\). Thus convergence to
zero alone does not justify comparison with a shrinking threshold, even
on the rank grid. For the actual statistic, the self-comparison gives
\(p_N\geq(2N+1)^{-1}\), so the choice \(v_N=(2N+1)^{-2}\) makes the
corresponding rank test never reject under any law.

\subsection{Properties of the score-adaptive threshold}
\label{app:s4:adaptive-threshold}

Under process independence, conditional on the entire \(X\)-path,
the shifted scores have the form
\(V_N(s)=H_{N,X}(T_Y^sY)\) for a measurable function \(H_{N,X}\),
while \(Y\) retains its stationary marginal law.
Hence \((V_N(s))_{s\in\mathbb Z}\) is conditionally stationary.
Applying the stationary-rank inequality under this conditional
law yields
\begin{equation}
    P\{p_N\leq u\mid X\}\leq 2u
    \quad P\text{-a.s.},\qquad 0<u\leq\tfrac12.
    \label{app:s4:null-calibration}
\end{equation}
Taking expectations gives \(P\{p_N\leq u\}\leq 2u\).
Here \(p_N\) is defined in \eqref{app:s4:rank}.
The auxiliary extension to all \(s\in\mathbb Z\) is used only
in this argument.
Lemma~\ref{app:s4:record} ensures that the rank actually computed
by the test is measurable with respect to \(Z_{1:n}\).

\begin{lemma}[Score-adaptive threshold]
\label{lem:score-adaptive-threshold}
Fix \(\alpha_0\in(0,1)\), and define
\[
 a(v)=\min\{\alpha_0/2,v/4\},\qquad
 u_N=a(V_N(0)),\qquad
 \varphi_n=\mathbf1\{p_N\leq u_N\}.
\]
Then the following hold.
\begin{enumerate}
 \item Under every independent stationary ergodic joint law,
 \(u_N\to0\) almost surely and \(P\{\varphi_n=1\}\to0\).
 \item Under every dependent stationary ergodic joint law,
 \[
 u_N\longrightarrow u_P:=\min\{\alpha_0/2,\Delta_P/4\}>0
 \quad\text{almost surely}.
 \]
 If additionally \(p_N\to0\) in probability, then
 \(P\{\varphi_n=1\}\to1\).
 \item The bound \(0\leq u_N\leq\alpha_0/2\) holds for every sample.
 Under every independent stationary law, without an ergodicity assumption,
 \(P\{\varphi_n=1\}\leq\alpha_0\) for every \(n\).
\end{enumerate}
\end{lemma}

\begin{proof}
By Proposition~\ref{app:s4:original-limit}, joint stationarity and ergodicity
give \(V_N(0)\to\Delta_P\) almost surely, with \(\Delta_P=0\) under
independence and \(\Delta_P>0\) under dependence. Continuity of \(a\)
therefore proves both asserted limits of \(u_N\).

For finite-sample control, \(u_N\leq\alpha_0/2\) gives
\[
 P\{\varphi_n=1\}
 \leq P\{p_N\leq\alpha_0/2\}\leq\alpha_0,
\]
by \eqref{app:s4:null-calibration}. For \(n<6\), the test returns zero.

Under an ergodic null, for each deterministic
\(\varepsilon\in(0,1/2]\),
\[
 P\{\varphi_n=1\}
 \leq P\{u_N>\varepsilon\}+P\{p_N\leq\varepsilon\}
 \leq P\{u_N>\varepsilon\}+2\varepsilon.
\]
First let \(N\to\infty\), and then \(\varepsilon\downarrow0\).
This proves the first claim using only deterministic cutoffs in the
rank inequality.

Under a dependent ergodic law with \(p_N\to0\) in probability, put
\(\eta=u_P/2>0\). Since \(u_N\to u_P\) almost surely,
\[
 P\{\varphi_n=0\}
 =P\{p_N>u_N\}
 \leq P\{u_N<\eta\}+P\{p_N>\eta\}\longrightarrow0.
\]
Finally, \(N=\lfloor n/6\rfloor\to\infty\), so all asymptotic
conclusions hold along the full sequence of record lengths.
\end{proof}

\subsection{Truncation of the weighted covariance score}
\label{app:s4:covariance-truncation}

\begin{lemma}[From fixed-feature second moments to the weighted score]
\label{lem:weighted-score-truncation}
Let \(K_N\to\infty\), and define
\[
 e_{r,N}
 :=\frac1{2N+1}\sum_{s\in\mathcal S_N}
      \mathbb E_P|\widehat c_{r,N}(s)|^2.
\]
If \(e_{r,N}\to0\) for every fixed \(r\geq1\), then
\[
 \mathbb E_P S_N\longrightarrow0,
 \qquad
 S_N:=\frac1{2N+1}\sum_{s\in\mathcal S_N}V_N(s).
\]
Consequently, \(S_N\to0\) in \(L^1(P)\) and in probability.
\end{lemma}

\begin{proof}
Cauchy--Schwarz's inequality, applied under the product of \(P\)
and the uniform distribution on \(\mathcal S_N\), gives
\[
 \frac1{2N+1}\sum_{s\in\mathcal S_N}
       \mathbb E_P|\widehat c_{r,N}(s)|
 \leq \sqrt{e_{r,N}}.
\]
Since the features are indicators, their empirical covariance
satisfies \(|\widehat c_{r,N}(s)|\leq1/4\), and hence
\(\sqrt{e_{r,N}}\leq1/4\).

Fix \(R\geq1\). Since \(K_N\to\infty\), eventually \(K_N\geq R\).
By linearity of expectation and the preceding bounds,
\begin{align}
 \mathbb E_P S_N
 &\leq \sum_{r=1}^{K_N}2^{-r}\sqrt{e_{r,N}}
 \notag\\
 &\leq
 \sum_{r=1}^{R}2^{-r}\sqrt{e_{r,N}}
 +\frac14\sum_{r=R+1}^{K_N}2^{-r}
 \notag\\
 &\leq
 \sum_{r=1}^{R}2^{-r}\sqrt{e_{r,N}}
 +\frac14\,2^{-R}.
 \label{app:s4:truncation-bound}
\end{align}
For fixed \(R\), the finite sum tends to zero as \(N\to\infty\).
Therefore
\[
 0\leq\limsup_{N\to\infty}\mathbb E_P S_N
 \leq\frac14\,2^{-R}.
\]
Letting \(R\to\infty\) proves \(\mathbb E_P S_N\to0\).
Since \(S_N\geq0\), this is convergence in \(L^1(P)\);
convergence in probability follows from Markov's inequality.
\end{proof}

\subsection{Centering the empirical covariance}
\label{app:s4:centering}

Fix a feature index \(r\) for the remainder of the analytic argument, and
center \emph{both} indicators:
\begin{equation}
 f=F_r-\mathbb E_PF_r,\qquad g=G_r-\mathbb E_PG_r.
 \label{app:s4:centered}
\end{equation}
We have \(|f|,|g|\leq1\) and \(\mathbb E_Pf=\mathbb E_Pg=0\).
Set
\begin{align}
 \overline f_N&=\frac1{2N}\sum_{t\in\mathcal I_N}f(T_X^tX),&
 \overline g_N(s)&=\frac1{2N}\sum_{t\in\mathcal I_N}g(T_Y^{t+s}Y),
 \notag\\
 A_N(s)&=\frac1{2N}\sum_{t\in\mathcal I_N}
                f(T_X^tX)g(T_Y^{t+s}Y).
 \label{app:s4:long-average}
\end{align}
Expanding \(F_r=f+\mathbb E_PF_r\) and
\(G_r=g+\mathbb E_PG_r\) in \eqref{app:s4:empirical} cancels the two
linear terms and the constant term, giving the exact identity
\begin{equation}
 \widehat c_{r,N}(s)=A_N(s)-\overline f_N\,\overline g_N(s).
 \label{app:s4:centering-identity}
\end{equation}

\begin{lemma}[The centering correction is uniformly negligible]
\label{app:s4:centering-lemma}
If \(P\) is stationary and ergodic, then
\begin{equation}
 \sup_{s\in\mathbb Z}
 \mathbb E_P|\widehat c_{r,N}(s)-A_N(s)|^2
 \leq\mathbb E_P|\overline f_N|^2\longrightarrow0.
 \label{app:s4:uniform-centering}
\end{equation}
In particular,
\begin{equation}
 \left|
 \frac1{2N+1}\sum_{s=-N}^{N}\mathbb E_P|\widehat c_{r,N}(s)|^2
 -\frac1{2N+1}\sum_{s=-N}^{N}\mathbb E_P|A_N(s)|^2
 \right|\longrightarrow0.
 \label{app:s4:same-energy}
\end{equation}
\end{lemma}

\begin{proof}
Equation~\eqref{app:s4:moving-window} and centeredness give
\(\overline f_N\to0\) almost surely and in \(L^2(P)\). Since
\(|\overline g_N(s)|\leq1\) for every \(s\), the exact identity
\eqref{app:s4:centering-identity} proves
\eqref{app:s4:uniform-centering} without a uniform ergodic theorem for the
shifted \(g\)-averages. Furthermore, \(|A_N(s)|\leq1\) and
\(|\widehat c_{r,N}(s)|\leq1/4\), so
\begin{align*}
 \left|\mathbb E_P|\widehat c_{r,N}(s)|^2
             -\mathbb E_P|A_N(s)|^2\right|
 &\leq\mathbb E_P\bigl[
 (|\widehat c_{r,N}(s)|+|A_N(s)|)
 |\widehat c_{r,N}(s)-A_N(s)|\bigr]\\
 &\leq2\,\|\overline f_N\|_{L^2(P)}.
\end{align*}
This bound is independent of \(s\). Averaging over the shifts proves
\eqref{app:s4:same-energy}.
\end{proof}

\subsection{Replacing a long average by short blocks}
\label{app:s4:blocks}

For a positive integer \(H\), define
\begin{equation}
 B_H(s)=\frac1H\sum_{t=0}^{H-1}f(T_X^tX)g(T_Y^{t+s}Y).
 \label{app:s4:short-average}
\end{equation}
The order of limits will always be: keep \(H\) fixed, let \(N\to\infty\),
and then let \(H\to\infty\).

\begin{lemma}[A uniform block bound]
\label{app:s4:block-lemma}
For every stationary joint law, every fixed \(H\), and every \(N\) with
\(2N\geq H\),
\begin{equation}
 \mathbb E_P|A_N(s)|^2
 \leq\mathbb E_P|B_H(s)|^2+\frac{H-1}{2N},
 \qquad s\in\mathbb Z.
 \label{app:s4:block-bound}
\end{equation}
\end{lemma}

\begin{proof}
Write \(2N=qH+\ell\), where \(q\geq1\) and \(0\leq\ell<H\), and
put \(a_N=2N+1\). The \(q\) full blocks in \(\mathcal I_N\) have
averages
\[
 C_j(s)=\frac1H\sum_{i=0}^{H-1}
 f(T_X^{a_N+jH+i}X)g(T_Y^{a_N+jH+i+s}Y)
       =B_H(s)\circ T^{a_N+jH},\qquad 0\leq j<q.
\]
Let \(C(s)=q^{-1}\sum_{j=0}^{q-1}C_j(s)\). Joint stationarity and
Jensen's inequality give
\[
 \mathbb E_P|C(s)|^2
 \leq\frac1q\sum_{j=0}^{q-1}\mathbb E_P|C_j(s)|^2
 =\mathbb E_P|B_H(s)|^2.
\]
If \(\ell>0\), let \(D(s)\) be the average of the remaining \(\ell\)
products. Then \(|D(s)|\leq1\), and
\[
 A_N(s)=\frac{qH}{2N}C(s)+\frac{\ell}{2N}D(s).
\]
Convexity of squared modulus implies
\[
 \mathbb E_P|A_N(s)|^2
 \leq\frac{qH}{2N}\mathbb E_P|C(s)|^2
       +\frac{\ell}{2N}\mathbb E_P|D(s)|^2
 \leq\mathbb E_P|B_H(s)|^2+\frac{\ell}{2N}.
\]
For \(\ell=0\) the same bound holds with the remainder omitted. Since
\(\ell\leq H-1\), \eqref{app:s4:block-bound} follows. Its boundary
term is uniform over the relative shift \(s\).
\end{proof}

\subsection{Vanishing averaged covariance, score, and rank under KNR}
\label{app:s4:average-score}
\label{app:s4:product-limit}
\label{app:s4:spectral}
\label{app:score_limits}

\begin{proposition}[Vanishing shifted background]
\label{app:s4:covariance-vanishes}
\label{app:s4:score-vanishes}
\label{app:s4:rank-limit}
For every stationary ergodic KNR law,
\[
 e_{r,N}:=\frac1{2N+1}\sum_{s=-N}^{N}
       \mathbb E_P|\widehat c_{r,N}(s)|^2\longrightarrow0
 \quad\text{for each fixed }r,\qquad
 \mathbb E_PS_N\longrightarrow0,
\]
where \(S_N=(2N+1)^{-1}\sum_{s=-N}^{N}V_N(s)\).
If the law is dependent, then \(p_N\to0\) in probability.
\end{proposition}

\begin{proof}
Keep \(H\) fixed. Expanding the square and applying the two-sided ergodic
theorem to \(g_tg_u\), followed by bounded convergence, gives
\begin{align}
 \lim_{N\to\infty}\frac1{2N+1}\sum_{s=-N}^{N}
        \mathbb E_P|B_H(s)|^2
 &=\frac1{H^2}\sum_{t,u=0}^{H-1}
        \mathbb E_P[f_tf_u]\,\mathbb E_P[g_tg_u]\notag\\
 &=\left\|\frac1H\sum_{t=0}^{H-1}
      (U_X\otimes U_Y)^t(f\otimes g)\right\|_{L^2(Q)}^2,
 \qquad Q=P_X\otimes P_Y.
 \label{app:s4:fixed-block-limit}
\end{align}
Indeed, for each fixed \(t,u\),
\((2N+1)^{-1}\sum_{s=-N}^{N}g_{t+s}g_{u+s}\to\mathbb E_P[g_tg_u]\)
almost surely. The two-sided version follows by applying Birkhoff's
theorem to \(T\) and \(T^{-1}\). There are only \(H^2\) bounded terms,
and the two resulting marginal expectations factor under \(Q\).

Joint ergodicity implies marginal ergodicity. KNR therefore makes the
product system under \(Q\) ergodic by the common-eigenvalue criterion. By the mean ergodic theorem, the average in
\eqref{app:s4:fixed-block-limit} converges in \(L^2(Q)\) to
\(\mathbb E_Q[f(X)g(Y)]\mathbf1=0\), since both features are centered.
Combining \eqref{app:s4:same-energy}--\eqref{app:s4:fixed-block-limit},
for every fixed \(H\), yields
\[
 0\leq\limsup_{N\to\infty}e_{r,N}
 \leq\left\|\frac1H\sum_{t=0}^{H-1}
       (U_X\otimes U_Y)^t(f\otimes g)\right\|_{L^2(Q)}^2.
\]
Now let \(H\to\infty\). This proves \(e_{r,N}\to0\) with the required
order of limits: first \(N\to\infty\), then \(H\to\infty\).

Having established \(e_{r,N}\to0\) for every fixed \(r\),
Lemma~\ref{lem:weighted-score-truncation} yields
\(\mathbb E_P S_N\to0\) and hence \(S_N\to0\) in probability.

Finally, nonnegativity gives \(p_NV_N(0)\leq S_N\) pathwise. Under
dependence, \(\Delta_P>0\) and \(V_N(0)\to\Delta_P\) almost surely by
Proposition~\ref{app:s4:original-limit}. For every \(\varepsilon>0\),
\[
 P\{p_N>\varepsilon\}
 \leq P\{V_N(0)<\Delta_P/2\}
       +\frac{2\mathbb E_PS_N}{\varepsilon\Delta_P}\longrightarrow0.
\]
\end{proof}

\section{Proof details for Section 5}
\label{app:section5}

Fix \(\alpha\notin\mathbb Q\), put \(\lambda=e^{2\pi i\alpha}\), and let
\(\varphi_n:(\{0,1\}^2)^n\to[0,1]\) satisfy
\(\mathbb E_P\varphi_n\to1\) for every
\(P\in\mathcal C_\lambda\).
The parameters \(\alpha\) and \(\lambda\) remain fixed throughout this appendix. The value of a randomized test is its conditional rejection probability;
all arguments below therefore apply to randomized and deterministic tests.
We work throughout this appendix with the null class
\[
\mathcal N\cap\mathcal E_{\mathrm{TE}}
=\{P\in\mathcal E_{\mathrm{TE}}:P=P_X\otimes P_Y\}.
\]
We regard phases as elements of \(\mathbb R/\mathbb Z\), with Haar
probability \(m\), and set
\[
 R_\omega u=u+\omega,\qquad b(u)=\mathbf1_{[0,1/2)}(u),\qquad
 d_{\mathbb S^{1}}(u,v)=\min_{r\in\mathbb Z}|u-v-r|.
\]
Here \(d_{\mathbb S^{1}}\) denotes the circular distance used for the frequency
parameters in Section 5. Throughout,
\(d_{\mathrm{TV}}(\rho,\sigma)=\sup_E|\rho(E)-\sigma(E)|\).

\subsection{Reduction to mixing null laws}

\begin{lemma}
\label{app:s5:mixing-reduction}
Independent mixing laws are weakly dense in
\(\mathcal N\cap\mathcal E_{\mathrm{TE}}\).
Consequently, to prove density of a subset of
\(\mathcal N\cap\mathcal E_{\mathrm{TE}}\), it suffices to show
that it intersects every relative neighborhood of every
independent mixing law.
\end{lemma}

\begin{proof}
We first approximate a stationary binary law \(\mu\) by mixing
laws. Let \(\eta\) be the i.i.d.\ Bernoulli\((1/2)\) law and put
\(\rho=(1-\varepsilon)\mu+\varepsilon\eta\), where
\(0<\varepsilon<1\).
For \(m\ge1\), define an order-\(m\) Markov chain by
\[
p(b\mid a_1,\ldots,a_m)
=
\frac{\rho[a_1,\ldots,a_m,b]}
     {\rho[a_1,\ldots,a_m]},
\qquad b\in\{0,1\}.
\]
Stationarity of \(\rho\) implies that its length-\(m\) block
distribution is stationary for the corresponding block-state
chain. All allowed transitions are positive, so this finite
chain is irreducible and aperiodic. Its stationary binary output
law \(\mu_{m,\varepsilon}\) is therefore mixing and has the same
length-\((m+1)\) block distribution as \(\rho\).

For each fixed block length, its distribution under
\(\mu_{m,\varepsilon}\) thus converges to that under \(\mu\)
as \(m\to\infty\) and \(\varepsilon\to0\).
Hence \(\mu_{m,\varepsilon}\Rightarrow\mu\).
Approximating both marginals of an independent law separately
and taking their products proves the asserted density, since
products of mixing systems are mixing and therefore totally
ergodic.

To justify the reduction, fix an arbitrary
\(P_0\in\mathcal N\cap\mathcal E_{\mathrm{TE}}\) and a relative
neighborhood \(\mathcal O\).
Choose an ambient weakly open neighborhood \(\mathcal G\) with
\[
P_0\in\mathcal G,
\qquad
\mathcal G\cap
(\mathcal N\cap\mathcal E_{\mathrm{TE}})
\subseteq\mathcal O.
\]
By density, there is an independent mixing law
\(\widetilde P_0\in\mathcal G\).
Since \(\mathcal G\) is also a neighborhood of
\(\widetilde P_0\), the local construction around
\(\widetilde P_0\) produces the required null law inside
\(\mathcal O\).
We may therefore use \(\widetilde P_0\) as the center and
relabel it \(P_0\).
\end{proof}

\subsection{Rotation extensions and eigenvalues visible after observation}
\label{app:s5:extensions}

\begin{lemma}[Exact spectrum of a mixing rotation extension]
\label{app:s5:mixing-spectrum}
Let \(S\) be an invertible mixing probability-preserving transformation.
For the prescribed irrational \(\alpha\), the product \(S\times R_\alpha\)
is totally ergodic and has point spectrum exactly \(\langle\lambda\rangle\).
The same conclusion holds after adjoining independent Bernoulli shifts.
\end{lemma}

\begin{proof}
A mixing transformation has no nonconstant Koopman eigenfunctions.
Indeed, a centered nonzero eigenfunction would have correlations of
constant nonzero modulus, contradicting mixing. If
\(F\circ(S\times R_\alpha)=\eta F\), expand in the circle coordinate:
\[
F(x,u)=\sum_{k\in\mathbb Z}F_k(x)e^{2\pi iku},\qquad
F_k\circ S=\eta\lambda^{-k}F_k.
\]
Some \(F_k\) is nonzero, so \(\eta\lambda^{-k}=1\) and that coefficient
is constant. Conversely, \(e^{2\pi iku}\) realizes \(\lambda^k\).
This proves the exact point-spectrum assertion. For every \(q\geq1\),
\(S^q\) is mixing and \(q\alpha\) is irrational. The same expansion
shows that every invariant function of \(S^q\times R_{q\alpha}\) is
constant, so the product is totally ergodic. Bernoulli shifts are mixing,
and finite products of mixing systems are mixing, as follows first for
tensor observables by factoring correlations and then by \(L^2\) density.
Absorb these shifts into \(S\) to obtain the last assertion.
\end{proof}

\begin{lemma}[Generic rotation extensions]
\label{app:s5:generic-extension}
Let \((\Xi,\rho,S)\) be an invertible totally ergodic system with separable
\(L^2(\rho)\). Outside a countable set of \(\omega\in\mathbb R/\mathbb Z\),
\(S\times R_\omega\) is totally ergodic. A product with an independent
two-sided Bernoulli shift, and every measurable factor, also preserve
total ergodicity.
\end{lemma}

\begin{proof}
Fix \(q\geq1\). The eigenvalue set \(\Sigma(S^q)\) is countable:
eigenvectors of a unitary operator for distinct eigenvalues are orthogonal,
and \(L^2(\rho)\) is separable. If \(F\in L^2(\rho\otimes m)\) is
invariant under \(S^q\times R_{q\omega}\), its Fourier coefficients
\[
 F_\ell(x)=\int F(x,u)e^{-2\pi i\ell u}\,dm(u)
\]
satisfy
\[
 F_\ell(S^qx)=e^{-2\pi i\ell q\omega}F_\ell(x).
\]
For \(\ell=0\), ergodicity of \(S^q\) makes \(F_0\) constant.
For \(\ell\ne0\), a nonzero coefficient requires
\(e^{-2\pi i\ell q\omega}\in\Sigma(S^q)\). For each fixed
\(q,\ell\) this excludes only countably many \(\omega\), since the map
\(\omega\mapsto e^{-2\pi i\ell q\omega}\) has finite fibers.
Outside the union of these exceptional sets over \(q\geq1\) and
\(\ell\ne0\), all nonconstant Fourier coefficients vanish. Parseval's
identity makes \(F\) constant, proving total ergodicity. Rational
frequencies can also be excluded; every nonempty open arc still contains
an admissible irrational frequency.

For the Bernoulli assertion, let \(B\) be the independent Bernoulli
shift and fix \(q\geq1\). The transformation \(B^q\) is mixing: this
holds first for cylinder functions by independence of disjoint coordinate
sets, and then for all \(L^2\) functions by density. For simple tensors,
\[
 \langle f_1\otimes g_1,(U_S\otimes U_B)^{qn}(f_2\otimes g_2)\rangle
 =\langle f_1,U_S^{qn}f_2\rangle\langle g_1,U_B^{qn}g_2\rangle.
\]
Decompose the \(g_i\) into their means and centered parts. The terms
with a centered Bernoulli part tend to zero; the remaining term has the
constant-product Ces\`aro limit by ergodicity of \(S^q\). Density of
finite sums of tensors and the mean ergodic theorem give ergodicity of
\(S^q\times B^q\). This holds for every \(q\).
Finally, an event invariant under a power of a factor transformation
pulls back to an event invariant under the same power of the original
transformation, proving the factor assertion.
\end{proof}

\begin{lemma}[Eigenvalues survive an observable projection]
\label{app:s5:projection}
Let \(U\) be the Koopman operator of an invertible probability-preserving
system, and let \(0\ne\zeta\in L^2\) satisfy \(U\zeta=\lambda\zeta\).
Let \(\mathcal G\) be a sub-\(\sigma\)-field invariant under both the
transformation and its inverse. Then
\[
 h:=\mathbb E[\zeta\mid\mathcal G]
 \quad\text{satisfies}\quad Uh=\lambda h.
\]
In particular, \(h\ne0\) if some \(Z\in L^2(\mathcal G)\) satisfies
\(\mathbb E[Z\overline\zeta]\ne0\).
Moreover, a jointly ergodic law whose coordinate processes share an
eigenvalue \(\lambda\ne1\) is dependent.
\end{lemma}

\begin{proof}
The closed subspace \(L^2(\mathcal G)\) is invariant under both \(U\)
and \(U^{-1}\), so its orthogonal projection commutes with \(U\).
Also \(\mathbb E[Z\overline\zeta]=\mathbb E[Z\overline h]\), proving
nonvanishing. For an observed path \(Z\), a
\(\sigma(Z_t:t\in\mathbb Z)\)-measurable \(h\) is a measurable
function of that path. Shift commutation therefore makes it an
eigenfunction on the observed marginal, not merely on the latent system.

For the last assertion, normalize coordinate eigenfunctions \(f(X)\)
and \(g(Y)\) for \(\lambda\) to have modulus one; their moduli are
constant by marginal ergodicity. The product \(f(X)\overline{g(Y)}\)
is invariant, hence is a constant \(a\) of modulus one by joint
ergodicity. Both eigenfunctions have mean zero because \(\lambda\ne1\).
Independence would consequently give
\(a=\mathbb E[f(X)\overline{g(Y)}]=0\), a contradiction.
\end{proof}

\subsection{The torus prototype and finite-horizon perturbations}
\label{app:prototype-kr-te}

Use the prescribed irrational \(\alpha\) fixed above.
For each fixed \(\beta\), let \(D_\beta^{\mathrm{rot}}\) be the law of
\[
 (A_t,B_t^{(\beta)})
 =\bigl(b(\Theta+t\alpha),b(\Theta+\beta+t\alpha)\bigr),
 \qquad \Theta\sim m.
\]
\begin{proposition}[The coded prototype]
\label{app:s5:prototype}
Every \(D_\beta^{\mathrm{rot}}\) belongs to \(\mathcal C_\lambda\) and is dependent. Its Haar average
\(\bar D^{\mathrm{rot}}:=\int D_\beta^{\mathrm{rot}}\,dm(\beta)\)
is the product of the two marginal path laws, but is not ergodic.
Furthermore,
\(\mathbb E_{\bar D^{\mathrm{rot}}}\varphi_n\to1\).
\end{proposition}

\begin{proof}
The irrational rotation is totally ergodic: for every \(q\geq1\),
its \(q\)th power has eigenfunctions \(e^{2\pi i\ell u}\), and
\(e^{2\pi i\ell q\alpha}=1\) only for \(\ell=0\). Its coded pair
is a shift-commuting factor, so is totally ergodic as well.
Put \(\zeta=e^{2\pi i\Theta}\) and \(\lambda=e^{2\pi i\alpha}\).
Then \(U\zeta=\lambda\zeta\), and
\begin{equation}
 \mathbb E[A_0\overline\zeta]
 =\int_0^{1/2}e^{-2\pi iu}\,du=-\frac{i}{\pi},\qquad
 \mathbb E[B_0^{(\beta)}\overline\zeta]
 =-\frac{i}{\pi}e^{2\pi i\beta}.
 \label{app:s5:first-fourier}
\end{equation}
Lemma~\ref{app:s5:projection} gives a nonzero \(\lambda\)-eigenfunction
in each observed coordinate. Each coordinate is a factor of
\(R_\alpha\), whose point spectrum is \(\langle\lambda\rangle\).
The normalized observed eigenfunctions and all their integer powers give
the reverse inclusion. Thus
\(\Sigma_X(D_\beta^{\mathrm{rot}})=\Sigma_Y(D_\beta^{\mathrm{rot}})
=\langle\lambda\rangle\), so the pair belongs to
\(\mathcal C_\lambda\) and is dependent.

Under independent Haar phases \(\Theta,\beta\), the map
\((\Theta,\beta)\mapsto(\Theta_X,\Theta_Y)=(\Theta,\Theta+\beta)\)
has law \(m\otimes m\). The two complete coded paths are therefore
independent. If \(f,g\) are normalized marginal
\(\lambda\)-eigenfunctions, then \(f(X)\overline{g(Y)}\) is an
invariant random variable of modulus one and mean zero under this
product law. It is nonconstant, proving nonergodicity.
For each \(n\), the map
\(\beta\mapsto\mathbb E_{D_\beta^{\mathrm{rot}}}\varphi_n\) is Borel:
it is the integral of a bounded jointly Borel function of the phase
and \(\beta\). The power assumption and dominated convergence give
\[
 \mathbb E_{\bar D^{\mathrm{rot}}}\varphi_n
 =\int\mathbb E_{D_\beta^{\mathrm{rot}}}\varphi_n\,dm(\beta)
 \longrightarrow1.
\]
No uniformity in \(\beta\) is required.
\end{proof}

Using independent Haar phases \(\Theta_X,\Theta_Y\), define
\[
 Q_\vartheta^{\mathrm{rot}}
 :=\mathcal L\bigl((b(\Theta_X+t\alpha),
                    b(\Theta_Y+t\vartheta))_{t\in\mathbb Z}\bigr).
\]
It is an independent stationary law for every \(\vartheta\), and
\(Q_\alpha^{\mathrm{rot}}=\bar D^{\mathrm{rot}}\).
If \(1,\alpha,\vartheta\) are linearly independent over \(\mathbb Q\),
every power of the torus translation by \((\alpha,\vartheta)\) is
ergodic. Indeed, its invariant characters would require
\(q(a\alpha+b\vartheta)\in\mathbb Z\) for some
\((a,b)\in\mathbb Z^2\setminus\{(0,0)\}\), which is excluded;
Fourier completeness then leaves only constants. Hence
\(Q_\vartheta^{\mathrm{rot}}\in\mathcal N\cap\mathcal E_{\mathrm{TE}}\).
The exceptional \(\vartheta\) form a countable set, so admissible
frequencies occur arbitrarily close to \(\alpha\).

\begin{lemma}[Finite-block total variation bound]
\label{app:tv-bound}
For every \(n\geq1\) and frequencies \(u,v\),
\begin{equation}
 d_{\mathrm{TV}}\bigl((Q_u^{\mathrm{rot}})_{1:n},
                       (Q_v^{\mathrm{rot}})_{1:n}\bigr)
 \leq n(n+1)d_{\mathbb S^{1}}(u,v).
 \label{app:s5:rotation-tv}
\end{equation}
For probability measures \(\rho,\sigma\) and a measurable
\(h\in[0,1]\),
\begin{equation}
 |\mathbb E_\rho h-\mathbb E_\sigma h|
 \leq d_{\mathrm{TV}}(\rho,\sigma).
 \label{app:s5:test-tv}
\end{equation}
\end{lemma}

\begin{proof}
Couple the two laws using the same initial phases. Their left coordinates
agree. At time \(t\geq1\), the right bits disagree with probability
\[
 m\bigl([0,1/2)\mathbin\triangle([0,1/2)-t(u-v))\bigr)
 =2d_{\mathbb S^{1}}(tu,tv)\leq2t\,d_{\mathbb S^{1}}(u,v).
\]
A union bound over \(1\leq t\leq n\) and the coupling inequality give
\eqref{app:s5:rotation-tv}. The coupling inequality itself follows from
\(|\mathbb P(V\in E)-\mathbb P(W\in E)|\leq\mathbb P(V\ne W)\).
Finally, integrating
\(\rho(h>s)-\sigma(h>s)\) over \(0\leq s\leq1\) proves
\eqref{app:s5:test-tv}.
\end{proof}

In particular, given \(k\geq1\) and \(c<1\), first choose
\(n\geq k\) with
\(\mathbb E_{\bar D^{\mathrm{rot}}}\varphi_n>(1+c)/2\).
Then choose a rationally independent \(\vartheta\) with
\(n(n+1)d_{\mathbb S^{1}}(\vartheta,\alpha)<(1-c)/2\).
The resulting \(Q_\vartheta^{\mathrm{rot}}\in\mathcal N\cap\mathcal E_{\mathrm{TE}}\) has
\(\mathbb E_{Q_\vartheta^{\mathrm{rot}}}\varphi_n>c\).
A fixed law within this same torus family is constructed in
Appendix~\ref{app:fixed-torus-null} below.

\subsection{Sparse embedding near an arbitrary independent law}
\label{app:s5:embedding}

For the construction in this subsection, let
\(P_0=\mu\otimes\nu\) be an independent mixing law and let
\((X^0,Y^0)\sim P_0\). Such backgrounds are dense in the admissible null
class by Lemma~\ref{app:s5:mixing-reduction}; the density argument below
first approximates an arbitrary null by one of them. Keep the prescribed
irrational \(\alpha\) fixed. Lemma~\ref{app:s5:mixing-spectrum} gives
total ergodicity of \(P_0\times R_\alpha\) and its exact point spectrum
\(\langle\lambda\rangle\).
Let \(\Theta\sim m\) be independent of \((X^0,Y^0)\).
Let \(M^X,M^Y\) be independent two-sided i.i.d.\
Bernoulli\((p)\) processes, independent of \((X^0,Y^0,\Theta)\),
where \(0<p<1\).
For each fixed \(\beta\), set
\begin{equation}
 \begin{aligned}
 X_t^{(\beta)}&=(1-M_t^X)X_t^0+M_t^Xb(\Theta+t\alpha),\\
 Y_t^{(\beta)}&=(1-M_t^Y)Y_t^0+M_t^Yb(\Theta+\beta+t\alpha),
 \end{aligned}
 \qquad t\in\mathbb Z,
 \label{app:s5:marker-definition}
\end{equation}
and let \(D_\beta\) be their joint path law.

\begin{proposition}[Exact marginal spectra after sparse substitution]
\label{app:s5:marker-eigenvalue}
For every \(p\in(0,1)\) and every \(\beta\),
\[
\Sigma_X(D_\beta)=\Sigma_Y(D_\beta)=\langle\lambda\rangle.
\]
In particular, \(D_\beta\in\mathcal C_\lambda\), and it
is dependent. Moreover, for any cylinder event \(C\) supported on a finite set of
time indices \(F\subset\mathbb Z\),
\begin{equation}
 |L(C)-P_0(C)|\leq2|F|p
 \label{app:s5:neighborhood-bound}
\end{equation}
for every law \(L\) formed by the same marker substitution, whatever
the input phases and frequencies. Thus the same bound holds for
\(D_\beta\), its Haar average, and every frequency-split law below.
\end{proposition}

\begin{proof}
The latent transformation is the product of \(P_0\times R_\alpha\)
with the two Bernoulli marker shifts. By
Lemma~\ref{app:s5:mixing-spectrum}, it is totally ergodic and its point
spectrum is exactly \(\langle\lambda\rangle\); the observation in
\eqref{app:s5:marker-definition} is a shift-commuting factor.

On the latent system put \(\zeta=e^{2\pi i\Theta}\) and
\(\lambda=e^{2\pi i\alpha}\). Independence and
\eqref{app:s5:first-fourier} give
\begin{equation}
 \mathbb E[X_0^{(\beta)}\overline\zeta]=-\frac{ip}{\pi},\qquad
 \mathbb E[Y_0^{(\beta)}\overline\zeta]
 =-\frac{ip}{\pi}e^{2\pi i\beta}.
 \label{app:s5:marker-fourier}
\end{equation}
The background terms vanish because \(\mathbb E\overline\zeta=0\).
Applying Lemma~\ref{app:s5:projection} to the two observed path
\(\sigma\)-fields gives nonzero marginal \(\lambda\)-eigenfunctions.
Every marginal eigenfunction lifts to a latent eigenfunction, so each
observed marginal spectrum is contained in \(\langle\lambda\rangle\).
Conversely, marginal ergodicity allows the surviving
\(\lambda\)-eigenfunctions to be normalized to modulus one; their integer
powers give every \(\lambda^k\). Hence both observed marginal spectra
equal \(\langle\lambda\rangle\). This proves membership in
\(\mathcal C_\lambda\), and
Lemma~\ref{app:s5:projection} proves dependence. In particular,
unobserved markers do not prevent the eigenvalue from being observable.

For the last assertion, couple the substituted law with \(P_0\) using
the same background. If none of the \(2|F|\) relevant markers fires,
the pairs agree on \(F\). The union bound and the coupling inequality
give \eqref{app:s5:neighborhood-bound}, independently of the input.
\end{proof}

On the finite-alphabet path space, cylinder probabilities generate the
weak topology because cylinders are clopen and continuous functions are
uniformly approximable by functions of finitely many coordinates.
Consequently, given an ambient weak neighborhood \(\mathcal G\) of \(P_0\),
choose cylinders \(C_1,\ldots,C_J\) with finite supports \(F_1,\ldots,F_J\)
and \(\eta>0\) such that
\[
 |L(C_\ell)-P_0(C_\ell)|<\eta \quad (1\leq\ell\leq J)
 \quad\Longrightarrow\quad L\in\mathcal G.
\]
By \eqref{app:s5:neighborhood-bound}, choosing \(p\in(0,1)\) with
\(2p\max_{1\leq\ell\leq J}|F_\ell|<\eta\) places all substituted laws
in \(\mathcal G\). If \(\mathcal O\) is a neighborhood relative
to \(\mathcal N\cap\mathcal E_{\mathrm{TE}}\), take \(\mathcal G\) with
\(P_0\in \mathcal G\) and \(\mathcal G \cap(\mathcal N\cap\mathcal E_{\mathrm{TE}})\subseteq\mathcal O\);
only the final admissible null laws need belong to \(\mathcal O\).

\subsection{Haar averaging, the substitution kernel, and frequency splitting}
\label{app:s5:phase-mixture}

Define
\[
 \bar D:=\int D_\beta\,dm(\beta).
\]
As in the prototype, \(\Theta_X=\Theta\) and
\(\Theta_Y=\Theta+\beta\) are independent Haar phases after averaging.
The complete latent collections \((X^0,M^X,\Theta_X)\) and
\((Y^0,M^Y,\Theta_Y)\) are independent. Thus \(\bar D\) has independent
coordinate processes. Parameterized integration is measurable for the
same reason as in the prototype, so the power assumption yields
\begin{equation}
 \mathbb E_{\bar D}\varphi_n
 =\int\mathbb E_{D_\beta}\varphi_n\,dm(\beta)\longrightarrow1.
 \label{app:s5:mixture-power}
\end{equation}
The averaged law \(\bar D\) is stationary as a mixture of stationary laws.
By translation invariance of Haar measure, the marginal path laws of
\(D_\beta\) are the same for every \(\beta\). Thus both marginals of
\(\bar D\) retain the common eigenvalue \(\lambda\ne1\) established in
Proposition~\ref{app:s5:marker-eigenvalue}. Since the coordinate processes
are independent under \(\bar D\), Lemma~\ref{app:s5:projection} implies
that the joint shift is nonergodic.

For \(\vartheta\in\mathbb R/\mathbb Z\), define \(Q_\vartheta\) by
\begin{equation}
 \begin{aligned}
 X_t&=(1-M_t^X)X_t^0+M_t^Xb(\Theta_X+t\alpha),\\
 Y_t&=(1-M_t^Y)Y_t^0+M_t^Yb(\Theta_Y+t\vartheta),
 \end{aligned}
 \label{app:s5:split-definition}
\end{equation}
where the two phases are independent Haar variables, independent of
the background and markers. For every \(\vartheta\), this defines
independent coordinate processes, and the averaged representation gives
the exact identity
\begin{equation}
 Q_\alpha=\bar D.
 \label{app:s5:mixture-identity}
\end{equation}

\begin{lemma}[The substitution kernel and total variation contraction]
\label{app:markov-kernel}
\label{app:markov kernel}
There is a Markov kernel \(K_{p,n}\), depending on \(P_0,p,n\) but
not on the rotation frequencies, such that
\[
 (Q_\vartheta)_{1:n}=(Q_\vartheta^{\mathrm{rot}})_{1:n}K_{p,n},
 \qquad
 \bar D_{1:n}=(\bar D^{\mathrm{rot}})_{1:n}K_{p,n}.
\]
Consequently,
\begin{equation}
 d_{\mathrm{TV}}\bigl((Q_\vartheta)_{1:n},\bar D_{1:n}\bigr)
 \leq n(n+1)d_{\mathbb S^{1}}(\vartheta,\alpha).
 \label{app:s5:embedded-tv}
\end{equation}
\end{lemma}

\begin{proof}
For a deterministic input \(z=(a_{1:n},b_{1:n})\), draw the background
block with law \((P_0)_{1:n}\) and the independent marker blocks.
Let \(K_{p,n}(z,\cdot)\) be the distribution of the output with bits
\((1-M_t^X)X_t^0+M_t^Xa_t\) and
\((1-M_t^Y)Y_t^0+M_t^Yb_t\). This is a Markov kernel on finite sets
and its definition contains no frequency. Conditioning on the rotation
input proves the two identities.
For any Markov kernel \(K\) and event \(E\), \(K(\cdot,E)\in[0,1]\);
hence \eqref{app:s5:test-tv} gives
\[
 |(\rho K)(E)-(\sigma K)(E)|
 \leq d_{\mathrm{TV}}(\rho,\sigma).
\]
Taking the supremum over \(E\) proves contraction.
Apply it to the two rotation-block laws and then use
Lemma~\ref{app:tv-bound} with \(v=\alpha\).
\end{proof}

\begin{lemma}[Admissible split frequencies]
\label{app:s5:split-te}
All but countably many \(\vartheta\) give \(Q_\vartheta\in\mathcal N\cap\mathcal E_{\mathrm{TE}}\).
In particular, such frequencies occur arbitrarily close to \(\alpha\).
\end{lemma}

\begin{proof}
The mixing background ensures that \(P_0\times R_\alpha\) is totally ergodic for the prescribed \(\alpha\), by Lemma~\ref{app:s5:mixing-spectrum}.
Apply Lemma~\ref{app:s5:generic-extension} again, now to this entire
system: outside a countable exceptional set,
\(P_0\times R_\alpha\times R_\vartheta\) is totally ergodic.
Adding the Bernoulli markers and passing to the observed factor preserve
this property. Coordinate independence was proved directly from the
latent collections, so the resulting law belongs to \(\mathcal N\cap\mathcal E_{\mathrm{TE}}\).
\end{proof}

\subsection{The Baire argument and the proof of the impossibility theorem}
\label{app:baire}

\begin{lemma}[The admissible null class is Baire]
\label{app:s5:baire-space}
The space \(\mathcal N\cap\mathcal E_{\mathrm{TE}}\), with its relative weak topology, is nonempty
and completely metrizable. For each \(n\),
\(P\mapsto\mathbb E_P\varphi_n\) is weakly continuous.
\end{lemma}

\begin{proof}
The finite-alphabet two-sided path space \(\Omega\) is compact
metrizable; its probability-measure space is compact metrizable for weak
convergence. The shift is continuous, so its invariant-law space
\(\mathcal S\) is closed and compact. Let
\[
 \mathcal S_0:=\{P\in\mathcal S:P=P_X\otimes P_Y\}.
\]
This stationary independence class is closed. Indeed, if
\(P_j\Rightarrow P\) with \(P_j\in\mathcal S_0\), then for any
coordinate cylinders \(A,B\), their clopenness gives
\[
 P(A\cap B)=\lim_jP_j(A\cap B)
 =\lim_jP_j(A)P_j(B)=P(A)P(B).
\]
Two applications of the \(\pi\)-\(\lambda\) theorem extend this to the
full coordinate path \(\sigma\)-fields.

To show that \(\mathcal E_{\mathrm{TE}}\) is a \(G_\delta\) in
\(\mathcal S\), choose a countable uniformly dense family
\((h_j)\subset C(\Omega;\mathbb R)\), and define
\[
 V_{q,N,j}(P):=\int\left|
 \frac1N\sum_{t=0}^{N-1}h_j\circ T^{qt}-\int h_j\,dP
 \right|^2dP.
\]
For fixed \(q,N,j\), this is continuous in \(P\), by expanding the
square into integrals of continuous functions. The mean ergodic theorem
gives a limit as \(N\to\infty\), equal to the squared distance between
the \(T^q\)-invariant projection of \(h_j\) and its constant mean.
Density of the \(h_j\) in \(L^2(P)\) implies that \(T^q\) is
ergodic exactly when this limit is zero for every \(j\). Since the
limit exists,
\[
 \mathcal E_{\mathrm{TE}}
 =\bigcap_{q,j,r,L\geq1}\ \bigcup_{N\geq L}
      \{P\in\mathcal S:V_{q,N,j}(P)<1/r\},
\]
a \(G_\delta\). Hence
\(\mathcal N\cap\mathcal E_{\mathrm{TE}}=\mathcal S_0\cap\mathcal E_{\mathrm{TE}}\) is a
\(G_\delta\) subset of a compact metrizable space, and is completely
metrizable. It is nonempty, for example by taking independent i.i.d.\
binary coordinate processes. Finally,
\[
 \mathbb E_P\varphi_n
 =\sum_{z\in(\{0,1\}^2)^n}\varphi_n(z)\,P\{Z_{1:n}=z\}
\]
is a finite linear combination of clopen cylinder probabilities, hence
is continuous even when the test is randomized.
\end{proof}

\begin{proposition}[Arbitrarily late large false-positive probabilities]
\label{app:s5:open-dense}
For every \(k\geq1\) and \(c\in(0,1)\), the set
\[
 \mathcal U_{k,c}:=\bigcup_{n\geq k}
        \{P\in\mathcal N\cap\mathcal E_{\mathrm{TE}}:\mathbb E_P\varphi_n>c\}
\]
is open and dense in \(\mathcal N\cap\mathcal E_{\mathrm{TE}}\).
\end{proposition}

\begin{proof}
Fix $k\geq 1$ and $c\in(0,1)$. Openness follows from
Lemma~\ref{app:s5:baire-space}.

For density, first fix an independent product law
$P_0=\mu\otimes\nu$ with mixing marginals and a relatively weakly
open neighborhood $\mathcal O$ of $P_0$ in
$\mathcal N\cap\mathcal E_{\mathrm{TE}}$. Choose an ambient weakly
open neighborhood $\mathcal G$ of $P_0$ such that
\[
 \mathcal G\cap(\mathcal N\cap\mathcal E_{\mathrm{TE}})\subseteq\mathcal O.
\]
Use $P_0$ as the background in the sparse construction and choose
$p>0$ by \eqref{app:s5:neighborhood-bound} so that every substituted
law lies in $\mathcal G$. Keep this $p$ and the prescribed $\alpha$ fixed.
By Proposition~\ref{app:s5:marker-eigenvalue}, every phase-locked
alternative has shared point spectrum exactly $\langle\lambda\rangle$.
By \eqref{app:s5:mixture-power}, choose $n\geq k$ such that
\[
 \mathbb E_{\bar D}\varphi_n>\frac{1+c}{2}.
\]

With $n$ fixed, choose an admissible $\vartheta$ from
Lemma~\ref{app:s5:split-te} sufficiently close to $\alpha$ that
\[
 n(n+1)d_{\mathbb S^{1}}(\vartheta,\alpha)<\frac{1-c}{2}.
\]
Then $Q_\vartheta\in \mathcal G\cap(\mathcal N\cap\mathcal E_{\mathrm{TE}})
\subseteq\mathcal O$. By \eqref{app:s5:test-tv}
and \eqref{app:s5:embedded-tv},
\[
 \mathbb E_{Q_\vartheta}\varphi_n
 \geq \mathbb E_{\bar D}\varphi_n
       -d_{\mathrm{TV}}\bigl((Q_\vartheta)_{1:n},\bar D_{1:n}\bigr)
 >\frac{1+c}{2}-\frac{1-c}{2}=c.
\]
Thus $\mathcal O\cap \mathcal U_{k,c}\ne\varnothing$.

By Lemma~\ref{app:s5:mixing-reduction}, every nonempty relatively
weakly open subset of $\mathcal N\cap\mathcal E_{\mathrm{TE}}$
contains an independent product law with mixing marginals. Every
such set can therefore serve as $\mathcal O$ in the construction
above and hence meets $\mathcal U_{k,c}$. This proves density in
$\mathcal N\cap\mathcal E_{\mathrm{TE}}$.
\end{proof}

\begin{proof}[Proof of the impossibility theorem in Section 3]
For \(j\geq2\), let \(c_j=1-1/j\). The preceding proposition and
the Baire theorem show that
\[
 \mathcal R:=\bigcap_{j\geq2}\bigcap_{k\geq1}\mathcal U_{k,c_j}
\]
is a dense \(G_\delta\) subset of the nonempty space \(\mathcal N\cap\mathcal E_{\mathrm{TE}}\).
In fact,
\[
 \mathcal R
 =\left\{Q\in\mathcal N\cap\mathcal E_{\mathrm{TE}}:
              \limsup_{n\to\infty}\mathbb E_Q\varphi_n=1\right\}
 =\mathcal B_\varphi,
\]
where \(\mathcal B_\varphi\) is the notation in the theorem. The equality
holds because membership in every \(\mathcal U_{k,c_j}\) means that rejection
probabilities exceed \(1-1/j\) at arbitrarily late horizons, for every
\(j\). Thus the asserted bad-null set is residual. Weak consistency
would require \(\mathbb E_Q\varphi_n\to0\) for every
\(Q\in\mathcal N\cap\mathcal E_{\mathrm{TE}}\), contradicting this nonempty set. This proves the
claimed impossibility for possibly randomized tests.
\end{proof}

\subsection{A fixed bad null within the torus family}
\label{app:fixed-torus-null}

The preceding Baire argument gives fixed null laws in the full class
\(\mathcal N\cap\mathcal E_{\mathrm{TE}}\). The following separate construction verifies the
stronger prototype claim that a fixed bad law can be chosen within
\(\{Q_\vartheta^{\mathrm{rot}}\}\).

\begin{proposition}[A fixed torus null]
\label{app:s5:fixed-torus}
For the prescribed irrational \(\alpha\), there exists
\(\vartheta^\star\) such that
\[
 Q_{\vartheta^\star}^{\mathrm{rot}}\in\mathcal N\cap\mathcal E_{\mathrm{TE}},
 \qquad
 \limsup_{n\to\infty}
    \mathbb E_{Q_{\vartheta^\star}^{\mathrm{rot}}}\varphi_n=1.
\]
The parameter \(\vartheta^\star\) may be chosen in any prescribed
nonempty open arc of the parameter circle.
\end{proposition}

\begin{proof}
Put \(r_n(\vartheta)=\mathbb E_{Q_\vartheta^{\mathrm{rot}}}\varphi_n\).
Lemma~\ref{app:tv-bound} makes \(r_n\) continuous.
The set
\[
 \mathcal G_\alpha:=\{\vartheta:
                   1,\alpha,\vartheta
                   \text{ are linearly independent over }\mathbb Q\}
\]
has countable complement, which we enumerate as \((e_j)_{j\geq1}\).

First we show that every open parameter interval contains a frequency
\(\vartheta_0\) for which \(r_n(\vartheta_0)\to1\).
The frequencies
\[
 \vartheta_0=\frac{\alpha+s}{d}\pmod1,
 \qquad d\geq1,\quad s\in\{0,\ldots,d-1\},
\]
are dense, since for fixed \(d\) they have spacing \(1/d\).
Fix one such frequency. For \(u\sim m\) and fixed \(\beta\), let
\(E_\beta\) be the law
\[
 \bigl(b(du+t\alpha),\ b(u+\beta+t\vartheta_0)\bigr)_{t\in\mathbb Z}.
\]
Because \(d\vartheta_0=\alpha+s\pmod1\), this pair is a factor of
the irrational rotation \(u\mapsto u+\vartheta_0\), hence is totally
ergodic. The left observable has nonzero Fourier coefficient
\[
 \int b(du)e^{-2\pi idu}\,dm(u)=-\frac{i}{\pi},
\]
so its path factor has eigenvalue \(e^{2\pi id\vartheta_0}
=e^{2\pi i\alpha}\), by Lemma~\ref{app:s5:projection}.
The right observable has a nonzero first Fourier coefficient,
so its path factor has eigenvalue
\(\eta=e^{2\pi i\vartheta_0}\).
Normalize a corresponding eigenfunction to modulus one.
Its \(d\)th power is then a nonzero eigenfunction with eigenvalue
\(\eta^d=e^{2\pi i\alpha}=\lambda\). The left path, being a function of
\(du+t\alpha\), is also a factor of \(R_\alpha\). Its spectrum is
therefore contained in \(\langle\lambda\rangle\), while its visible
\(\lambda\)-eigenfunction and integer powers give equality. The right
spectrum contains \(\langle\lambda\rangle\). Consequently,
\[
\Sigma_X(E_\beta)\cap\Sigma_Y(E_\beta)=\langle\lambda\rangle,
\]
so every \(E_\beta\) belongs to \(\mathcal C_\lambda\)
and is dependent. The right marginal may have additional eigenvalues;
the exact-spectrum condition concerns the intersection.

Averaging over independent Haar \(\beta\) makes \((du,u+\beta)\)
independent Haar phases: conditionally on \(u\), the second phase is
uniform, and \(du\) itself is uniform. Consequently,
\[
 \int E_\beta\,dm(\beta)=Q_{\vartheta_0}^{\mathrm{rot}},
 \qquad r_n(\vartheta_0)\longrightarrow1
\]
by the power assumption and dominated convergence.

Now choose a nondegenerate closed interval \(J_0\) inside the prescribed open arc, using a local coordinate on the circle, and put \(n_0=0\). Inductively, choose one of the above resonant \(\vartheta_0\) inside \(\operatorname{int}J_{j-1}\), and then \(n_j>n_{j-1}\) with \(r_{n_j}(\vartheta_0)>1-1/(j+1)\). Continuity gives an open subinterval of \(\operatorname{int}J_{j-1}\) on which this strict inequality holds. Inside it choose a nondegenerate closed interval \(J_j\) of diameter at most \(2^{-j}\), avoiding the finitely many points \(e_1,\ldots,e_j\). The nested intervals have a unique common point \(\vartheta^\star\). It avoids every \(e_j\), so lies in \(\mathcal G_\alpha\), and \[
 r_{n_j}(\vartheta^\star)>1-\frac1{j+1}.
\] Thus \(Q_{\vartheta^\star}^{\mathrm{rot}}\) is an independent totally ergodic law with rejection limsup one, as required.
\end{proof}

\section{Explanation for Section 6}
\label{app:kronecker}

The names \emph{Kronecker-nonresonant} (KNR) and \emph{Kronecker-resonant} (KR) describe whether the two coordinate processes share nontrivial rotation structure. For an ergodic probability-preserving system, the Kronecker factor is generated by its Koopman eigenfunctions, and its \(L^2\)-space is their closed linear span. The factor system is isomorphic to an ergodic rotation on a compact abelian group \citep{eisner2015operator}. Its eigenvalues therefore describe the frequencies present in this rotation structure. A shared eigenvalue \(\lambda\neq1\) corresponds to a nonconstant oscillatory mode present in both coordinate systems; we refer to this spectral overlap as \emph{resonance}.

The following proposition makes this terminology precise in terms of factors of the observed joint system. Fix \(P\in\mathcal E\), and let \(\mathcal K_X,\mathcal K_Y\) be the marginal Kronecker factors. Let \(\pi_X,\pi_Y\) be the coordinate projections, and write
\[
 \mathcal Z_1^X
 :=\overline{\pi_X^{-1}\mathcal K_X}^{\,P},
 \qquad
 \mathcal Z_1^Y
 :=\overline{\pi_Y^{-1}\mathcal K_Y}^{\,P},
\]
where both completions include all subsets of \(P\)-null sets in the ambient joint path space. Thus \(\mathcal Z_1^X,\mathcal Z_1^Y\) are the coordinate Kronecker factors viewed as sub-sigma-fields of the same joint space \((\Omega,P,T)\). The notation \(\mathcal Z_1\) agrees with the usual indexing of the Host--Kra factors \citep{HostKra2005}.

\begin{proposition}[Nonresonance and the intersection of coordinate Kronecker factors]
\label{app:kronecker:intersection}
For \(P\in\mathcal E\),
\[
 \Sigma_X(P)\cap\Sigma_Y(P)=\{1\}
 \quad\Longleftrightarrow\quad
 \mathcal Z_1^X\cap\mathcal Z_1^Y
 \text{ is trivial modulo }P.
\]
\end{proposition}

\begin{proof}
We identify coordinate functions with their pullbacks to \(\Omega\). Suppose first that \(\lambda\neq1\) is a common eigenvalue, and choose corresponding eigenfunctions \(e_X,e_Y\). Marginal ergodicity allows us to normalize them so that
\(|e_X|=|e_Y|=1\) almost surely. Their phase changes cancel:
\[
 (e_X\overline{e_Y})\circ T
 =(\lambda e_X)\overline{\lambda e_Y}
 =e_X\overline{e_Y}.
\]
Joint ergodicity therefore gives \(e_X\overline{e_Y}=c\) almost surely for a constant \(c\) of modulus one. Hence \(e_X=ce_Y\) almost surely. By completion, \(e_X\) is measurable with respect to both coordinate Kronecker factors. It is nonconstant because \(\lambda\neq1\), so their intersection is nontrivial.

Conversely, suppose that the intersection is nontrivial. Choose a nonzero centered function
\[
 h\in L^2(\mathcal Z_1^X\cap\mathcal Z_1^Y,P).
\]
Since \(L^2(\mathcal Z_1^X,P)\) is the closed linear span of the lifted \(X\)-coordinate eigenfunctions and \(h\) is orthogonal to constants, there are an eigenvalue \(\lambda\neq1\) and an associated \(X\)-coordinate eigenfunction \(e_X\) such that \(\langle h,e_X\rangle_{L^2(P)}\neq0\).

Let \(U\) be the joint Koopman operator. By the mean ergodic theorem applied to \(\lambda^{-1}U\),
\[
 v:=\lim_{M\to\infty}
 \frac1M\sum_{t=0}^{M-1}\lambda^{-t}U^t h
 \quad\text{in }L^2(P)
\]
is the orthogonal projection of \(h\) onto
\(\ker(U-\lambda\,\mathrm{id})\). In particular,
\[
 Uv=\lambda v,\qquad
 \langle v,e_X\rangle_{L^2(P)}
 =\langle h,e_X\rangle_{L^2(P)}\neq0.
\]
Both coordinate Kronecker factors are invariant under \(T\), and their \(L^2\)-spaces are closed. Thus every average above, and hence its limit \(v\), belongs to both spaces. The nonzero function \(v\) therefore defines a \(\lambda\)-eigenfunction in each marginal system. Consequently, \(\lambda\in\Sigma_X(P)\cap\Sigma_Y(P)\), with \(\lambda\neq1\).
\end{proof}

The proposition shows why the terminology refers specifically to the Kronecker factors: KR means that these coordinate factors share a nonconstant observable inside the ergodic joining, whereas KNR means that their intersection is trivial. Equivalently, KNR makes the independent product system under \(P_X\otimes P_Y\) ergodic \citep[Theorem~12(2)]{Quas2012}. It does not require either marginal Kronecker factor to be trivial, and it does not itself imply independence of the entire coordinate processes.


\end{document}